\documentclass[12pt]{amsart}

\usepackage{hyperref}
\usepackage{amssymb, amsthm, amsmath, graphicx, enumerate, mathptmx}

\usepackage{cleveref}  

\usepackage[all]{xy}

\numberwithin{equation}{section}

\newcommand{\N}{\mathbb{N}}

\newcommand{\Q}{\mathbb{Q}}
\newcommand{\R}{\mathbb{R}}

\newcommand{\Z}{\mathbb{Z}}
\newcommand{\cA}{\mathcal{A}}

\newcommand{\Fc}{\mathcal{F}}

\newcommand{\Hc}{\mathcal{H}}

\newcommand{\cL}{\mathcal{L}}
\newcommand{\Ls}{\mathcal{L}_{\rm std}}
\newcommand{\Lm}[1]{\mathcal{L}_{{\rm std}, #1}}

\newcommand{\cN}{\mathcal{N}}
\newcommand{\Pc}{\mathcal{P}}
\newcommand{\cP}{\mathcal{P}}
\newcommand{\cQ}{\mathcal{Q}}

\newcommand{\Ab}{{\bf A}}

\newcommand{\Fi}[1]{{\bf F}^{\FI, #1}}
\newcommand{\Fo}[1]{{\bf F}^{\FIO, #1}}

\newcommand{\Fb}{{\bf F}}
\newcommand{\Gb}{{\bf G}}

\newcommand{\Nb}{{\bf N}}
\newcommand{\Pb}{{\bf P}}
\newcommand{\Qb}{{\bf Q}}
\newcommand{\bP}{{\bf P}}

\newcommand{\Xb}{{\bf X}}
\newcommand{\XI}[1]{{\bf X}^{\FI, #1}}
\newcommand{\XO}[1]{{\bf X}^{\FIO, #1}}

\newcommand{\bF}{\mathbf{F}}
\newcommand{\Ib}{\mathbf{I}}
\newcommand{\Jb}{\mathbf{J}}
\newcommand{\M}{\mathbf{M}}
\newcommand{\x}{\mathbf{x}}

\DeclareMathOperator{\Alg}{Alg}

\DeclareMathOperator{\depth}{depth}

\DeclareMathOperator{\FI}{FI}

\DeclareMathOperator{\Hom}{Hom}

\DeclareMathOperator{\id}{id}

\DeclareMathOperator{\ini}{in}

\DeclareMathOperator{\lc}{lc}

\DeclareMathOperator{\lm}{lm}
\DeclareMathOperator{\lt}{lt}
\DeclareMathOperator{\Mod}{Mod}
\DeclareMathOperator{\Mon}{Mon}

\DeclareMathOperator{\FIO}{OI}
\DeclareMathOperator{\OI}{OI}
\DeclareMathOperator{\pd}{pd}

\DeclareMathOperator{\reg}{reg}
\DeclareMathOperator{\Res}{Res}

\DeclareMathOperator{\si}{si}

\DeclareMathOperator{\Sym}{Sym}

\DeclareMathOperator{\Tor}{Tor}

\DeclareMathOperator{\wi}{wi}

\DeclareMathOperator{\pnt}{\raise 0.5mm \hbox{\large\bf.}}
\DeclareMathOperator{\lpnt}{\hbox{\large\bf.}}

\newcommand{\la}{\langle}
\newcommand{\ra}{\rangle}

\newcommand{\s}{\; | \;}

\let\phi=\varphi
\newcommand{\eps}{\varepsilon}

\DeclareRobustCommand{\coprod}{\mathop{\text{\fakecoprod}}}
\newcommand{\fakecoprod}{%
  \sbox0{$\prod$}%
  \smash{\raisebox{\dimexpr.9625\depth-\dp0}{\scalebox{1}[-1]{$\prod$}}}%
  \vphantom{$\prod$}%
}
\newtheorem{thm}{\bf Theorem}[section]
\newtheorem{lem}[thm]{\bf Lemma}
\newtheorem{cor}[thm]{\bf Corollary}
\newtheorem{prop}[thm]{\bf Proposition}
\newtheorem{conj}[thm]{\bf Conjecture}

\theoremstyle{definition}
\newtheorem{defn}[thm]{\bf Definition}

\newtheorem{rem}[thm]{\bf Remark}
\newtheorem{ex}[thm]{\bf Example}
\newtheorem{alg}[thm]{\bf Algorithm}
\title[Equivariant Hilbert Series and Asymptotic Properties]{Rationality of Equivariant Hilbert Series and Asymptotic Properties}

\author[Uwe Nagel]{Uwe Nagel}
\address{Department of Mathematics, University of Kentucky, 715 Patterson Office Tower, Lexington, KY 40506-0027, USA}
\email{uwe.nagel@uky.edu}

\begin{document}

\begin{abstract}
An $\FI$- or an $\OI$-module $\M$ over a corresponding noetherian polynomial algebra $\Pb$ may be thought of as a sequence of compatible modules $\M_n$ over a polynomial ring $\Pb_n$ whose number of variables depends linearly on $n$. In order to study invariants of the modules $\M_n$ in dependence of $n$, an equivariant Hilbert series is introduced if $\M$ is graded. If $\M$ is also finitely generated, it is shown that this series is a rational function. Moreover, if this function is written in reduced form rather precise information about the irreducible factors of the denominator is obtained. This is key for applications. It follows that the Krull dimension of the modules $\M_n$ grows eventually linearly in $n$, whereas the multiplicity of $\M_n$ grows eventually exponentially in $n$. Moreover, for any fixed degree $j$, the vector space dimensions of the degree $j$ components of $\M_n$ grow eventually polynomially in $n$. As a consequence, any graded Betti number of $\M_n$ in a fixed homological degree and a fixed internal degree grows eventually polynomially in $n$. Furthermore, evidence is obtained to support a conjecture that the Castelnuovo-Mumford regularity and the projective dimension of $\M_n$ both grow eventually linearly in $n$. It is also shown that modules $\M$ whose width $n$ components $\M_n$ are eventually Artinian can be characterized by their equivariant Hilbert series. 
Using regular languages and finite automata, an algorithm for computing equivariant Hilbert series is presented.

\end{abstract}

\date{June 22, 2020}

\thanks{The author was partially supported by Simons Foundation grants \#317096 and \#636513. }

\maketitle



\section{Introduction}

$\FI$-modules over a field $K$ have been instrumental in exploring and establishing instances of  representation stability (see, e.g., \cite{CE, CEF, CEFN, PS, SS-14, SS-16}). 
In algebraic statistics, sequences $(\Ib_n)_{n \in \N}$ of symmetric ideals $\Ib_n$ in polynomial rings with increasingly many variables arise naturally (see \cite{AH, Draisma-factor,   Draisma, HS} and also \cite{DK,GS, S-16, S-17, SS-12b} for related results). In \cite{NR2}, $\FI$-modules over an $\FI$-algebra as well as their ordered analogs, $\OI$-modules, have been introduced to capture aspects of both approaches. In this article we mainly consider graded $\FI$- and $\OI$-modules over a noetherian polynomial $\FI$- or $\OI$-algebra $\Pb$. Intuitively, one can think of such an $\FI$-module as a sequence $(\M_n)_{n \in \N}$ of compatible symmetric modules $\M_n$ over  $\Pb_n$, where each $\Pb_n$ is a polynomial ring over a field $K$ whose number of variables grows linearly in $n$. For example,  fix a partition $\lambda$ with $c$ parts $\lambda_1 \ge \cdots \ge \lambda_c \ge 1$ and an integer $k$ with $1 \le k \le c+1$, 
and let  $\Ib_n$  be the ideal of $K[x_{i, j} \; \mid \; 1 \le i \le c, \ 1 \le j \le n]$ that is generated by the $k$-minors of an $(c+1) \times n$ matrix 
\[
\begin{bmatrix}
1 & 1  & \ldots & 1 \\
x_{1,1}^{\lambda_1} & x_{1, 2}^{\lambda_1} & \ldots & x_{1, n}^{\lambda_1} \\
\vdots \\
x_{c,1}^{\lambda_c} & x_{c, 2}^{\lambda_c} & \ldots & x_{c, n}^{\lambda_c}
\end{bmatrix}. 
\]
Each ideal $\Ib_n$ is symmetric, that is, invariant under the action of the symmetric group $\Sym(n)$ acting on column indices, i.e., $\sigma \cdot x_{i, j} = x_{i, \pi (j)}$. These ideals (along with suitable maps) fit together to form an $\FI$-ideal of a polynomial $\FI$-algebra $\Pb$. 

$\OI$-modules are more general than $\FI$-modules. For example, if $\Ib$ is an $\FI$-ideal of $\Pb$, each ideal $\Ib_n$ is invariant under an action of a symmetric group whose size depends on $n$. In contrast, \emph{every}  ideal $J$ in a polynomial ring with finitely many variables generates an $\OI$-ideal $\Ib$ with $J = \Ib_{n_0}$ for some integer $n_0$ (see \Cref{ex:any ideal generated OI-ideal}).  For large $n$, one may consider the ideals  $\Ib_n$ as asymptotic shadows of $J$.  

In \cite{NR2}, it was shown that any finitely generated $\FI$- or $\OI$-module over $\Pb$ is noetherian, i.e., any submodule is again finitely generated. This finiteness result begs the question if quantitative invariants of the modules $\M_n$, considered as $\Pb_n$-modules, are also finite in the sense that one can predict the value for each sufficiently large $n$, provided one knows the values for finitely many modules $\M_n$. For polynomial rings over a field $K$ with finitely many variables, Hilbert pioneered an approach by considering generating functions given by vector space dimensions of graded components. These are called Hilbert series nowadays. Thus, we consider a graded $\FI$- or $\OI$-module $\M$, where each module $\M_n = \oplus_{j \in \Z} [\M_n]_j$ is $\Z$-graded with graded components $[\M_n]_j$. In the spirit of \cite{NR}, we define an equivariant Hilbert series of such a graded $\FI$- or $\OI$-module $\M$ as a formal power series in two variables
\[
H_{\M} (s, t)  = \sum_{n \ge 0, j \in \Z} \dim_K [\M_n]_j s^n t^j.  
\]
We show that it is a rational function of a particular form (see \Cref{thm:shape Hilb f.g. OI-module} for a more precise statement). 

\begin{thm}
     \label{thm-intro:shape Hilb f.g. OI-module}
If $\M$ is a finitely generated graded $\OI$-module over $\Pb$ that is trivial in negative degrees, then its equivariant Hilbert series is of the form
\[
H_{\M} (s, t) = \frac{g(s, t)}{(1-t)^a \cdot \prod_{j =1}^b [(1-t)^{c_j} - s \cdot f_j (t)]},
\]
where $a, b, c_j$ are non-negative integers, \ $g (s, t) \in \Z[s, t]$, and each $f_j (t)$ is a polynomial in $\Z[t]$ satisfying $f_j (1) >  0$ and $f_j (0) = 1$. 
\end{thm}

As a consequence, an analogous result is also true for finitely generated graded $\FI$-modules (see \Cref{cor:shape Hilb f.g. FI-module}). Notice that the assumption about triviality in negative degrees is harmless as it can be achieved by shifting the grading of $\M$. Such a degree shift changes the equivariant Hilbert series by factor equal to a power of $t$. If $\Pb$ is the smallest noetherian polynomial $\OI$- or $\FI$-algebra that is not equal to $K$ in every width,  then 
\Cref{thm-intro:shape Hilb f.g. OI-module} can be considerably strengthened (see \Cref{cor: hilb for c=1}). 

In the special case, where $\M$ is a quotient of $\Pb$, the above result was shown (for the most part) in \cite{NR}. Later, rationality of the equivariant Hilbert series in that case was proved by a different method in \cite{KLS}. Here we combine both approaches and introduce a novel technique. In fact, producing a suitable formal language, we show rather quickly rationality of the Hilbert series. Establishing the given factorization of its denominator as a product of irreducible polynomials requires considerably more work. The main new tool is a (local) decomposition result for monomial $\OI$-modules (see \Cref{prop:decomposition}). It is not functorial, but powerful enough to enable an induction on quotients of finitely generated free $\OI$-modules. 

The above description of the denominator in \Cref{thm-intro:shape Hilb f.g. OI-module} is crucial for deriving various consequences. The Krull dimension of $\M_n$ is given by a linear function in $n$ for $n \gg 0$, whereas the multiplicity of $\M_n$ eventually grows exponentially in $n$. In particular, the limits 
 \[
\lim_{n \to \infty} \frac{\dim \M_n}{n} \quad \text{ and } \quad  \lim_{n \to \infty} \sqrt[n]{\deg \M_n} > 0 
 \]
 exist and are equal to  integers (see \Cref{thm:growth invariants}). This suggest to define the first limit as the dimension of $\M$ and the second limit as the multiplicity of $\M$. 
While  $\lim_{n \to \infty} \frac{\dim \M_n}{n}$ is bounded above by $\lim_{n \to \infty} \frac{\dim \Pb_n}{n}$, there is no universal upper bound for $\lim_{n \to \infty} \sqrt[n]{\deg \M_n}$ (see \Cref{exa:asymptotic degree}). 
 Note that the result about Krull dimensions is also true for modules that are not necessarily graded (see \Cref{thm:growth of dim}).

In contrast to the exponential growth of multiplicity, if one fixes a degree $j$, then the vector space dimension of $[\M_n]_j$ grows only 
polynomially in $n$ (see \Cref{thm:polynomiality in fixed degree}  for a more precise statement). 

\begin{thm}
          \label{thm-intro:polynomial in fixed degree}
If $\M$ is a finitely generated graded $\OI$-module over $\Pb$, then, for any fixed integer $j$, there is a polynomial $p(t) \in \Q[t]$ such that $\dim_K [\M_n]_j = p(n)$ whenever $n \gg 0$. 
\end{thm}

This result extends \cite[Corollary 7.1.7]{SS-14}, which considers the case of an $\OI$-module over a field.

If $M$ is a finitely generated graded module over a polynomial ring in finitely many variables, then 
$M$ is Artinian if and only if its Hilbert series is a polynomial. There is an analogous result for 
$\OI$-modules. In fact, modules $\M$ with the property that $\M_n$ is Artinian for $n \gg 0$ can be 
characterized by their 
equivariant Hilbert series (see \Cref{thm:hilb artinian}). 

Our results have also consequences for graded Betti numbers. For any fixed $n$, the polynomial ring $\Pb_n$ is noetherian, and thus $\M_n$ has a finite graded minimal free resolution over $\Pb_n$. The generators of the free modules appearing in the resolution determine the graded Betti numbers 
\[
\beta_{i, j}^{\Pb_n} (\M_n) = \dim_K [\Tor_i^{\Pb_n} (\M_n, K)]_j. 
\]
The set of Betti numbers can be conveniently displayed in the Betti table of $\M_n$ whose $(i, j)$-entry is $\beta_{i, i+j}^{\Pb_n} (\M_n)$, that is,  column $j$ corresponds to homological degree $j$. Its entries in row $i$ refer to minimal syzygies of degree $i+j$. Thus, the number of columns in the Betti table of $\M_n$ is equal to the projective dimension $\pd_{\Pb_n} (\M_n)$ of $\M_n$ and the number of rows determines the Castelnuovo-Mumford regularity  $\reg (\M_n)$. 
For fixed $n$, Betti tables of finitely generated graded $\Pb_n$-modules are classified up to rational multiples in \cite{ES}. Here we are interested in the sequence of Betti tables for the modules $\M_n$ as $n$ varies. Given $\M$ and fixing any integer $j \ge 0$, there are only finitely many rows in column $j$ in which the entries of the Betti table of any module $\M_n$ can possibly be non-zero (see \cite[Theorem 7.7]{NR2}). \Cref{thm-intro:polynomial in fixed degree} implies that the value of any  $(i, j)$-entry in the Betti table of $\M_n$ is eventually given by a polynomial, that is, for any integers $i, j$, there is a polynomial $p(t) \in \Q[t]$ such that $\beta_{i, j}^{\Pb_n} (\M_n) = p (n)$ if $n \gg 0$ (see \Cref{thm:Betti numbers}). Our results also lead us to conjecture that the heights and the widths of the Betti tables of the modules $\M_n$, that is, $\pd_{\Pb_n} (\M_n)$ and $\reg (\M_n)$,  grow eventually linearly in $n$ (see \Cref{sec:concluding remarks} for further discussion). 
\smallskip

We now discuss the organization of this paper. In \Cref{sec:preliminaries} we fix notation and review needed results. Then we show rationality of the Hilbert series of a monomial submodule $\M$ of a free $\OI$-module with one generator in \Cref{sec:language monomial submod}. To this end we construct a formal language whose words are in bijection to the monomials of $\M$ and show that the language is regular. Maintaining this set-up, \Cref{sec:denominator}  is devoted to  establishing a central result. We obtain rather precise information on the irreducible factors of the denominator polynomial when the rational function of an equivariant Hilbert series is written in reduced form. A key step is the mentioned (local) decomposition result for monomial $\OI$-modules (see \Cref{prop:decomposition}). Using the theory of Gr\"obner bases, in \Cref{sec:general case}  this is extended to any finitely generated $\OI$- and $\FI$-module. In particular, we establish Theorems \ref{thm-intro:shape Hilb f.g. OI-module} and \ref{thm-intro:polynomial in fixed degree} there. 

In \Cref{sec:artinian mod}, we consider the special case of an Artinian $\OI$-module $\M$. By definition, this means that $\M_n$ is an Artinian $\Pb_n$-module if $n \gg 0$. The equivariant Hilbert series of an Artinian $\OI$-module is described in \Cref{thm:hilb artinian}). Then we 
establish 
the mentioned polynomial change of any graded Betti number $\beta_{i, j}^{\Pb_n} (\M_n)$  as $n$ varies and is sufficiently large. 

Finally, in \Cref{sec:concluding remarks} we first observe that all the above results are true in virtually greater generality, where $K$ is a noetherian standard graded algebra over a field $k$ and one uses vector space dimensions over $k$, that is, one considers the equivariant Hilbert series $\sum_{n \ge 0, j \in \Z} \dim_k  [\M_n]_j s^n t^j$. This is achieved by reducing to the case discussed above (see \Cref{thm:base ring}). 
Then we 
present a finite algorithm for computing the equivariant Hilbert series of a graded $\OI$-module. This algorithm utilizes the regularity of the language considered in  \Cref{sec:language monomial submod} in one of its steps. We conclude with offering and discussing conjectures about the 
projective dimension and the Castelnuovo-Mumford regularity of the modules $\M_n$ as $n$ varies (see Conjectures \ref{conj:reg growth}  and \ref{conj: pd growth}). 


\section{$\OI$-modules and Hilbert series}
\label{sec:preliminaries}

We introduce  notation, discuss examples and recall results that will be used throughout this paper. 

We consider two combinatorial categories. Denote by $\FI$ the category whose objects are finite sets and whose morphisms are injections (see \cite{CEF} for more details).
The category $\FIO$ is the subcategory of $\FI$ whose objects are totally ordered finite sets and whose morphisms are order-preserving injective maps (see \cite{SS-14}). It will be enough to work with the skeletons of these categories. 

For an integer $n \ge 0$, we set $[n] = \{1,2,\ldots,n\}$. Thus, $[0] = \emptyset$. We denote by $\N$ and $\N_0$ the set of positive integers and non-negative integers, respectively.
The category $\FIO$ is equivalent to its skeleton,  the category with objects $[n]$ for $n \in \N_0$ and morphisms being order-preserving injective maps $\eps\colon [m] \to [n]$. In particular, this implies $\eps (m) \ge m$. Similarly, the skeleton of $\FI$ is the subcategory with objects $[n]$ for $n \in \N_0$ and morphisms being injective maps $\eps\colon [m] \to [n]$.
In order to define a functor $\FIO \to C$ or $\FI \to C$ it is enough to define it on the corresponding skeleton. We will use this convention throughout the paper.

Let $K$ be a commutative ring with unity and 
denote by $K$-$\Alg$ the category of commutative, associative, unital $K$-algebras whose 
morphisms are $K$-algebra homomorphisms that map the identity of the domain onto the identity of the codomain. An \emph{$\FIO$-algebra over $K$}
 is a covariant functor $\Ab$ from $\FIO$ to the category $K$-$\Alg$ with $\Ab({\emptyset}) = K$.
Similarly, 
an \emph{$\FI$-algebra over $K$} is  a functor $\Ab$ from the category $\FI$ to $K$-$\Alg$ with $\Ab({\emptyset}) = K$ (see \cite[Definition 2.4]{NR2}). 
Since $\FIO$ is a subcategory of $\FI$, any $\FI$-algebra may also be considered as an $\FIO$-algebra. We often will use the same symbol to denote both of these algebras.

For an interval $[n]$, we write $\Ab_n$ for the $K$-algebra  $\Ab ([n])$ and refer to its elements as the elements of \emph{width $n$} in $\Ab$. Given a morphism $\eps\colon [m] \to [n]$, we often write $\eps_*\colon \Ab_m \to \Ab_n$ for the morphism $\Ab(\eps)$.

An $\FIO$-algebra $\Ab$ (or $\FI$-algebra over $K$, respectively)  is said to 
\emph{finitely generated}, if there exists
a finite subset $G \subset \coprod_{n\geq 0} \Ab_n$
which is not contained in any proper subalgebra of $\Ab$.
As in the classical case, finite generation can be characterized using polynomial $\OI$- or $\FI$-algebras as introduced in \cite[Definition 2.17]{NR2}. 

\begin{defn}[\cite{NR2}]
       \label{def:polynomial-FOI-algebra}
Let $d \ge 0$ be an integer.
\begin{enumerate}
\item
Define a functor $\XO{d} = \XO{d}_K\colon \FIO \to K$-$\Alg$ by letting
\[
\XO{d}_n = K\bigl[x_{\pi} \; : \; \pi \in \Hom_{\FIO} ([d], [n])\bigr]
\]
be the polynomial ring over $K$ with variables $x_{\pi}$,
and, for $\eps \in \Hom_{\FIO} ([m], [n])$, by defining
 \[
\XO{d} (\eps)\colon \XO{d}_m \to \XO{d}_n
 \]
as the $K$-algebra homomorphism given by mapping $x_{\pi}$ onto $x_{\eps \circ \pi}$.

A \emph{polynomial $\FIO$-algebra over $K$} is an $\FIO$-algebra that is isomorphic to a tensor product
 $\Xb = \bigotimes_{\lambda \in \Lambda} \XO{d_{\lambda}}$, where each $\Xb_n$ is a tensor product of rings
 $\XO{d_{\lambda}}_n$ over $K$.
\item
Ignoring orders, we similarly define an $\FI$-algebra $\XI{d}$ over $K$ and a \emph{polynomial $\FI$-algebra over $K$}.
\end{enumerate}
\end{defn}

According to \cite[Proposition 2.19]{NR2}, an $\OI$-algebra is finitely generated if and only if 
 there is a surjective natural transformation $\XO{d_1} \otimes_K \cdots \otimes_K \XO{d_k} \to \Ab$ for some integers $d_1,\ldots,d_k \ge 0$. 
 
Notice that, for every integer $n \ge 0$ the algebras $\XI{1}_n$ and $\XO{1}_n$ are isomorphic to a polynomial ring in $n$ variables over $K$, whereas $\XI{0}_n = \XO{0}_n = K$. Thus, we refer to $\XI{0}_n$ and $\XO{0}_n$ as algebras with constant coefficients. 

We will mostly consider algebras $(\XO{1})^{\otimes c}$ or $(\XI{1})^{\otimes c}$, where $c$ is an integer $c > 0$. For these, we have the following interpretations that we use throughout this paper. 

\begin{rem}
\label{rem:poly-alg}
Fix an integer $c \ge 1$. 
\begin{enumerate}

\item
Identifying $\eps \in \Hom_{\FIO} ([d], [n[)$ with its image $s_1 < s_2 < \cdots < s_d$ in $[n]$, we get for $\Pb = (\XO{1})^{\otimes c}$, 
 \[
 \Pb_n = K[x_{i, j}  \; \mid \; i \in [n], \ j \in [c] ]  
 \]
and, for each $\eps \in \Hom_{\OI} ([m], [n])$, a $K$-algebra homomorphism 
\[
\eps_* = \Pb (\eps): \Pb_m \to \Pb_n  \; \text{ defined by } \eps_* (x_{i, j}) = x_{i, \eps (j)}. 
\]

\item
Ignoring orders, one has analogous identifications for $(\XI{1})^{\otimes c}$. 
\end{enumerate}
\end{rem}

Assigning every variable in $(\XO{1})^{\otimes c}_n$ or $ (\XI{1})^{\otimes c}_n$ degree one, they become  polynomial rings with standard ($\Z$)-grading and each homomorphism $\eps$ is graded of degree zero. Thus, $(\XO{1})^{\otimes c}$ and $(\XI{1})^{\otimes c}$ are examples of a graded $\OI$- or $\FI$-algebra as defined in \cite[Remark 2.20]{NR2}. We will always use this standard grading. 
\smallskip

Now we consider modules. Denote by $K$-$\Mod$ the category of $K$-modules, and let 
 $\Ab$ be an $\FIO$-algebra over $K$.   Following \cite[Definition 3.1]{NR2}, the objects of the category of \emph{$\FIO$-modules over $\Ab$}, denoted $\FIO$-$\Mod (\Ab)$, are covariant functors $\M \colon \FIO \to K$-$\Mod$ such that,
\begin{itemize}
\item[(i)] for any integer $n \ge 0$, the $K$-module $\M_n = \M([n])$ is an $\Ab_n$-module, \ and

\item[(ii)]
for any morphism  $\eps\colon [m] \to [n]$ and any $a \in \Ab_m$, the following diagram is commutative
\begin{equation*}
\xy\xymatrixrowsep{1.8pc}\xymatrixcolsep{1.8pc}
\xymatrix{
\M_m \ar @{->}[r]^-{\M (\eps)} \ar @{->}[d]_-{\cdot a} & \M_n \ar @{->}[d]^-{\cdot \Ab(\eps )(a)} \\
\M_m \ar @{->}[r]^-{\M (\eps)} & \M_n.
}
\endxy
\end{equation*}
Here the vertical maps are given by multiplication by the indicated elements.
\end{itemize}
The morphisms of $\FIO$-$\Mod (\Ab)$ are natural transformations $F \colon \M \to \Nb$ such that, for every integer $n \ge 0$, the map $\M_n \stackrel{F(n)}{\longrightarrow} \Nb_n$ is an $\Ab_n$-module homomorphism.

Ignoring orders, we define similarly the category $\FI$-$\Mod (\Ab)$ of \emph{$\FI$-modules over an $\FI$-algebra $\Ab$}.
Its objects are functors from $\FI$ to $K$-$\Mod$ and its morphisms are natural transformations satisfying conditions analogous to those above.

In the case of constant coefficients, these concepts specialize to previously studied objects. Indeed,  if $\Ab = \XI{0}$ is the ``constant'' $\FI$-algebra over $K$ the category $\FI$-$\Mod (\Ab)$ is exactly the category of $\FI$-modules over $K$ as, for example, studied in \cite{CE, CEF, CEFN}. 
If $\Ab = \XO{0}$ is the ``constant'' $\OI$-algebra then $\FIO$-$\Mod (\Ab)$ is the category of 
$\OI$-modules as introduced in \cite{SS-14}.

The categories $\FIO$-$\Mod (\Ab)$ and $\FI$-$\Mod (\Ab)$ inherit the structure of an abelian category from $K$-$\Mod$, with all concepts such as subobject, quotient object, kernel, cokernel, injection, and surjection being defined ``pointwise'' from the corresponding concepts in $K$-$\Mod$ (see \cite[A.3.3]{W}). 

If $\M$ is any $\FIO$-module we often write $q \in \M$ instead  of $q \in \coprod_{n \ge 0} \M_n$   and refer to $q$ as an \emph{element} of $\M$. Similarly a \emph{subset} of $\M$ is defined as a subset of $\coprod_{n \ge 0} \M_n$. An element $q$ of $\M$ has \emph{width} $n$ if $q \in \M_n$. 

For a subset $E$ of any $\FIO$-module $\M$, we denote by
$\la E \ra_{\M}$ or simply $\la E \ra$ the smallest $\FIO$-submodule of $\M$ that contains $E$. It is called the
$\FIO$-submodule \emph{generated by $E$}. 

An $\FIO$-module $\M$ over an $\OI$-algebra $\Ab$ is said to be \emph{noetherian} if every $\FIO$-submodule of $\M$ is finitely generated.
The algebra $\Ab$ is \emph{noetherian} if it is a noetherian $\FIO$-module over itself.
Analogously, we define a \emph{noetherian $\FI$-module} and a \emph{noetherian $\FI$-algebra}.
In contrast to the classical situation, not all polynomial algebras are noetherian. In fact, if $K$ is a noetherian ring, then 
an algebra $\XO{d}$ or $\XI{d}$ is \emph{noetherian} if and only $d \in \{0, 1\}$ (see \cite[Proposition 4.8]{NR2}). Furthermore, for any integer $c \ge 1$, the algebras $(\XO{1})^{\otimes c}$ and $(\XI{1})^{\otimes c}$ are noetherian, provided $K$ is noetherian (see \cite[Corollaries 6.18 and 6.19]{NR2}). This finiteness property extends to modules. If $K$ is noetherian, then any finitely generated $\OI$-module over $(\XO{1})^{\otimes c}$ is noetherian. Similarly, any finitely generated $\FI$-module over $(\XI{1})^{\otimes c}$ is noetherian (see \cite[Theorem 6.17]{NR2}). 

Finite generation of an $\OI$-module can be characterized by using free modules as introduced in \cite[Definition 3.16]{NR2}. 

\begin{defn}[\cite{NR2}] 
\label{def:free}
\
\begin{enumerate}
\item
For an $\FIO$-algebra $\Ab$ over $K$ and an integer $d \ge 0$, let $\Fo{d} = \Fo{d}_{\Ab}$ be the $\FIO$-module over $\Ab$ defined by
\[
\Fo{d}_m = \oplus_{\pi \in  \Hom_{\FIO} ([d], [m])} \Ab_m e_{\pi} \cong (\Ab_m )^{\binom{m}{d}},
\]
where $m$ is any non-negative integer, and
\[
\Fo{d}(\eps)\colon \Fo{d}_m \to \Fo{d}_n, \ a e_{\pi} \mapsto \eps_*(a) e_{\eps \circ \pi},
\]
where $a \in \Ab_m$ and $\eps \in  \Hom_{\FIO} ([m], [n])$.

A \emph{free $\FIO$-module} over $\Ab$ is an $\FIO$-module that is isomorphic to
$\bigoplus_{\lambda \in \Lambda} \Fo{d_{\lambda}}$.
\item
Ignoring orders, we similarly define an $\FI$-module $\Fi{d}$ over an $\FI$-algebra $\Ab$ and a \emph{free $\FI$-module} over $\Ab$.
\end{enumerate}
\end{defn} 

Notice that $\Fo{d}$ and $\Fi{d}$ are generated by one element in width $d$, namely $e_{\id_{[d]}}$.  
An $\OI$-module $\M$ over $\Ab$ is finitely generated if and only if there is a surjective natural transformation 
$\bigoplus_{i = 1}^k \Fo{d_{i}} \to \M$ 
for some integers $d_i \ge 0$.
An analogous statement is true for any $\FI$-module $M$ over an $\FI$-algebra $\Ab$ (see \cite[Proposition 3.18]{NR2}). 

A \emph{$\Z$-graded $\FIO$-module} is an $\FIO$-module $\M$ over a graded $\FIO$-algebra $\Ab$  such that every $\M_n$ is a graded $\Ab_n$-module and every map $\M (\eps)\colon \M_n \to \M_p$ is a graded homomorphism of degree zero.
We will refer to it simply as a graded $\FIO$-module. Similarly, we define a graded $\FI$-module. 
If $\Ab$ is a graded $\OI$-algebra, every free $\OI$-module $\Fo{d}$ becomes a graded module over $\Ab$ by assigning its generator  $e_{\id_{[d]}}$ any degree. We will always set its degree to zero. It follows that, for each $\pi \in  \Hom_{\FIO} ([d], [n])$,  the element $e_{\pi} \in \Fo{d}_n$ has degree zero as well. 

Any finitely generated graded module over a noetherian polynomial ring gives rise to a graded $\OI$-module. 

\begin{ex}
     \label{ex:any ideal generated OI-ideal} 
(i) Consider any homogeneous ideal $J$ of a polynomial ring $P = K[y_1,\ldots,y_c]$ in $c$ variables. Let $\{f_1,\ldots,f_r\}$ be a generating set of $J$ consisting of homogeneous polynomials. Define $\Pb = (\XO{1})^{\otimes c}$ and polynomials $q_1,\ldots,q_r \in \Pb_1 = K[x_{1,1},\ldots,x_{c, 1}]$ by substituting $x_{i, 1}$ for $y_i$, that is, $q_j = f_j(x_{1,1},\ldots,x_{c, 1})$. Then the set $E = \{q_1,\ldots,q_r\}$ generates an $\OI$-ideal  $\Ib$ of $\Pb$ with $P/J \cong \Pb_1/\Ib_1$. Note that $E$ generates an $\FI$-ideal  $\Ib$ of $\Pb = (\XI{1})^{\otimes c}$ if and only if $J$ is invariant under the action of a symmetric group with $c$ letters permuting the indices of the $y$-variables. 

(ii) If $N$ is any finitely generated graded $P$-module, then an analogous construction gives an $\OI$-module $\M$ over $\Pb$  with $\M_1 = N$. Thus, one may think of invariants of the modules $\M_n$ with $n \gg 0$ as asymptotic versions of the corresponding invariants of $N$. 
\end{ex}

Note that the above construction can easily be varied. For example, in the setting of (i) above,  choosing positive integers $n_0, c'$ with $n_0 c' \ge c$ one gets analogously an $\OI$-ideal $\Ib_n$ of $\Pb = (\XO{1})^{\otimes c'}$ such that $\Pb_{n_0}/\Ib_{n_0}$ is isomorphic to a polynomial ring over $P/J$. 

The next observation shows how one may consider a polynomial $\OI$- or $\FI$-algebra over a standard
graded $k$-algebra essentially as a polynomial algebra over $k$, where $k$ is any field. 

\begin{ex}
    \label{exa:reduction to field}
Let $K$ be a noetherian standard graded algebra over a field $k$, that is, $K$ is isomorphic to 
$k[y_1,\ldots,y_r]/J$, where $J$ is an ideal generated by homogeneous polynomials $f_1,\ldots,f_s \in k[y_1,\ldots,y_r]$ with positive degrees and each $y_i$ has degree one. 

(i)
Let $\Ab =  (\XO{1}_K)^{\otimes c}$ be a polynomial $\OI$-algebra over $K$. We will see that $\Ab$ is an $\OI$-module over $\Pb = (\XO{1}_k)^{\otimes (c+r)}$. Indeed, let $\Ib$ be the ideal of $\Pb$ that is generated by the linear polynomials $x_{i, 1} - x_{i, 2} \in \Pb_2$ with $c < i \le c+r$ and $f_i (x_{c+1, 1},\ldots,x_{c+r, 1}) \in \Pb_1$ with $1 \le i \le s$. Then $\Ab_0 = K$ and $(\Pb/\Ib)_n = k$. However, for any $n \in \N$, there is a graded isomorphism of $k$-algebras 
$\Ab_n \cong (\Pb/\Ib)_n$. 

(ii) Similarly, one gets for $\FI$-algebras  a graded isomorphism $( (\XI{1}_K)^{\otimes c})_n \cong 
( (\XI{1}_k)^{\otimes (c+r}/\Ib)_n$ for every $n \in \N$, where $\Ib$ is the $\FI$-ideal generated by the same polynomials as in (i). 
\end{ex}

\smallskip 

Let $P$ be a polynomial ring in finitely many variables over any field $K$. The \emph{Hilbert function} of a finitely generated graded $P$-module $M$ in degree $j$ is  $h_{M}(j)=\dim_K [M]_j$, where we denote by $[M]_j$ the degree $j$ component of $M$.  It is well-known that,  for large $j$, this is actually a polynomial function in $j$. Equivalently, the Hilbert series of $M$ is a rational function. Recall that the \emph{Hilbert series} of $M$ is the formal power series
\[
H_M (t) = \sum_{j \in \Z} h_M (j) \cdot t^j.
\]
By Hilbert's theorem (see, e.g., \cite[Corollary 4.1.8]{BH}), if $M$ is not zero then this series can 
be uniquely written in the form
\[
 H_M(t)=\frac{g_M(t)}{(1-t)^d}\quad\text{with }\ g_M(t)\in\Z[t,t^{-1}]\ \text{ and }\ g_M(1) \ne 0.
\]
The number $d$ is the \emph{(Krull) dimension} of $M$ and $\deg M = g_M (1) > 0$ is the  \emph{degree} or \emph{multiplicity} of $M$. 

Consider now a polynomial $\FI$- or $\OI$-algebra $\Pb$ over a field $K$ and 
 a finitely generated graded $\FI$- or $\OI$-module $\M$ over $\Pb$.   Every $\Pb_n$-module $\M_n$ has a rational Hilbert series. Combining these we define the \emph{equivariant Hilbert series} of $\M$ as a formal power series in two variables
\[
H_{\M} (s, t) = \sum_{n \ge 0} H_{\M_n} (t) s^n  = \sum_{n \ge 0, j \in \Z} \dim_K [\M_n]_j s^n t^j.  
\]

Let us determine this series in the special case of a free $\OI$-module over a noetherian polynomial  algebra. 

\begin{prop}
    \label{prop:hilb free OI-mod}
Consider a free $\OI$-module $\Fb = \Fo{d}$ over $\Pb = (\XO{1})^{\otimes c}$. Its equivariant Hilbert series is 
\[
H_{\Fb} (s, t) =  \frac{s^d (1-t)^c}{[(1-t)^c - s]^{d+1}}. 
\]
\end{prop}

\begin{proof}
For any $n \in \N_0$, the polynomial ring $\Pb_n$ has $n c$ variables. Hence, its Hilbert series is 
$H_{\Pb_n} (t) = \frac{1}{(1-t)^{n c}}$. Since $\Fb_n$ has rank $\binom{n}{d}$ as a free $\Pb_n$-module, it follows
\begin{align*}
H_{\Fb} (s, t) & = \sum_{n \ge 0} \binom{n}{d} \frac{1}{(1-t)^{n c}} s^n \\
& =\left  (\frac{s}{(1-t)^c} \right )^d \cdot  \sum_{k \ge 0} \binom{d+k}{k}  \left (\frac{s}{(1-t)^c} \right )^k \\
& = \left  (\frac{s}{(1-t)^c} \right )^d \cdot  \frac{1}{(1 - \frac{s}{(1-t)^c})^{d+1}} \\ 
& = \frac{s^d (1-t)^c}{[(1-t)^c - s]^{d+1}}, 
\end{align*}
where we use a binomial series for the third equality. 
\end{proof}


\section{Monomial Submodules}
\label{sec:language monomial submod} 

We will enumerate monomials in a monomial submodule using words in a suitable formal language. 
By showing that the language is regular it follows that monomial submodules have a rational equivariant Hilbert series. The bijection between words and monomials developed in this section will also be used in the following section where we establish more detailed information about rational functions describing equivariant Hilbert series. Moreover, the regularity of the describing language is used later in the final section for developing an algorithm to compute equivariant Hilbert series. 

Throughout this section we fix a positive integer $c$ and
 consider modules over the graded  $\OI$-algebra $\Pb = (\XO{1})^{\otimes c}$, where $K$ is  any commutative ring. Thus, for any integer $m \ge 0$, 
\[
\Pb_m = K[x_{i, j} \; : \; i \in [c], \ j \in [m]].
\]
We always use the standard grading in which every variable $x_{i, j}$ has degree one. 

Furthermore, we fix an integer $d \ge 0$ and consider the free $\OI$-module $\bF = \Fo{d}$ over $\Pb$. In width $m$ it is 
\[
\Fo{d}_m = \oplus_{\pi} \Pb_m e_{\pi} \cong (\Pb_m)^{\binom{m}{d}},
\]
where the sum is taken over all $\pi \in \Hom_{\FIO} ([d], [m])$. Note that $\Fb$ is generated by one element, namely $e_{\id_{[d]}}$. We set its degree to zero. This induces a grading on $\Fb$ and turns $\Fb$ into a graded $\OI$-module over $\Pb$ with $[\Fb]_j = 0$ if $j  < 0$. 

A \emph{monomial} in $\Fb = \Fo{d}$ is an
element of some $\Fb_m$ of the form
\[
x^u  e_{\pi} = x_{\lpnt, 1}^{u_1} \cdots x_{\lpnt, m}^{u_m} e_{\pi}, \quad \text{ where } \pi \in \Hom_{\FIO} ([d], [m]), \; u_j \in \N_0^c,  
\]
 the $i$-th entry of $u_j \in \N_0^c$ is the exponent of the variable $x_{i, j}$, and $x_{\lpnt, j}^{u_j}$ is the product of these powers. Thus, $x^u = x_{\lpnt, 1}^{u_1} \cdots x_{\lpnt, m}^{u_m}$ is a monomial in $\bP_m$.  Denote by $\Mon (\bP)$ and $\Mon (\M)$ the set of monomials of $\bP$ and of a submodule $\M$ of $\bF$, respectively.  
A \emph{monomial submodule} of $\Fo{d}$ is an $\OI$-submodule that is generated by monomials. 
 Since $\Fb = \Fo{d}$ is fixed in this section, we write  $\la E \ra$ instead of $\la  E\ra_{\Fb}$ for the submodule generated by a subset $E$ of $\Fb$. 

Set $\N_0^0 = \emptyset$ and define shifting operators $T_0,\ldots,T_d\colon \Mon (\bP) \times \N_0^d \to \Mon (\bP) \times \N_0^d$ by 
\begin{align}
     \label{eq:def T}
T_i (x_{k, l}, (p_1,\ldots,p_d)) = (x_{k, l+1}, (q_1,\ldots,q_d)), 
\end{align}
where $T_0$ acts as the identity on the second component whereas, for $i = 1,\ldots,d$, 
\[
q_j = \begin{cases}
p_j & \text{if } 1 \le j <i \\
p_j + 1 & \text{if } i \le j \le d, 
\end{cases}
\]
and Equation \eqref{eq:def T} is extended multiplicatively in the first component for $i = 0,\ldots,d$. For example, one gets $T_1 ((x_{4, 2}^6, (5,5,5)) = (x_{4, 3}^6, (6,6,6))$, $T_3 ((x_{4, 2}^6, (5,5,5)) = (x_{4, 3}^6, (5,5,6))$ and $T_2 (1, (5,5,5)) = (1, (5,6,6))$.

Consider a set $\Sigma = \{\xi_1,\ldots,\xi_c, \tau_0,\ldots,\tau_d\}$ with $c+d+1$ elements, and let 
$\Sigma^*$ be the free monoid on $\Sigma$. In other words, $\Sigma^*$ consists of strings of 
elements of $\Sigma$. A 
\emph{formal language} with words in the alphabet $\Sigma$ is any subset of $\Sigma^*$. We refer to the elements of $\Sigma$ as letters. 
An interested reader may consult, e.g., \cite{HU} for an introduction to formal languages. 

Define  recursively a map 
\begin{align}
     \label{eq:def nu}  
\eta = (\eta_1, \eta_2)\colon \Sigma^* \to \Mon (\bP) \times \N_0^d   
\end{align}
by 
\begin{itemize}
\item[(i)] $\eta (e) = (1, (0,\ldots,0))$, where $e$ is the empty word; 

\item[(ii)] $\eta (\xi_i w) = x_{i,1} \cdot \eta (w)$; and  

\item[(iii)] $\eta (\tau_i w ) = T_i (\eta (w))$. 
\end{itemize}
Thus, $\eta (w)$ is obtained by replacing each $\xi_i$ in $w$ by $x_{i, 1}$ and each $\tau_i$ by $T_i$ applied to the string following it.

Denote the set of strictly increasing maps $\pi\colon [d] \to \N$ by $\Hom_{\OI} ([d], \N)$. 
For ease of notation, we will identify a map $\pi \in \Hom_{\OI} ([d], \N)$ 
with the $d$-tuple 
$(\pi(1) < \pi (2) < \cdots < \pi (d))$ and write $x^u (p_1,\ldots,p_d)$ instead of $(x^u, (p_1,\ldots,p_d))$. For example, one computes  $\eta (\xi_1^j \tau_1^{m}) = x_{1,1}^j  (m,m,\ldots,m)$ and 
$\eta (\tau_0^{m-1} \xi_1^j ) = x_{1,m}^j  (0, 0,\ldots,0)$. 

The map $\eta$ is neither surjective nor injective. For example, $(4, 3, 6)$ is not in the image of $\eta$ and  $\eta (\xi_1 \xi_2)  =  \eta (\xi_2 \xi_1)$ as the variables $x_{i, j}$ commute. We are seeking a bijection induced by $\eta$ by restricting its domain and codomain. As a first step we introduce a suitable sublanguage of $\Sigma^*$. We say that a word $w$ in 
$\Sigma^*$ is \emph{standard} if every substring $\xi_i \xi_j$ in $w$ satisfies $i \le j$.  Denote the set 
of standard words in $\Sigma^*$ by $\Sigma^*_{\rm std}$. Furthermore, for each $j =0,1,\ldots,d$,  define a  language 
 $\cL_j = \{\xi_1,\ldots,\xi_c, \tau_j\}^*$ and set $\overline{\cL_j} =\{ \{\xi_1,\ldots,\xi_c\}^* \tau_j \}^*$.  
Note that every non-empty word in $\overline{\cL_j}$  ends with the letter $\tau_j$. Furthermore, one has $\cL_j = \overline{\cL_j} \{\xi_1,\ldots,\xi_c\}^*$. 

We now define a language $\cL \subset \Sigma^*$ by the formula
\[
\cL = \cL_1 \tau_1 \cL_2 \tau_2 \ldots \cL_d \tau_d \overline{\cL_0}, 
\]
and denote  by $\Ls$ the set of standard words in $\cL$, that is, 
$\Ls = \cL \cap \Sigma^*_{\rm std}$. For an integer $m \ge 0$, let $\Lm{m}$ be the set of words in $\Ls$ that contain precisely $m$ $\tau$-letters. Note that, by definition of $\cL$, every word in $\cL$ contains at least $d$ $\tau$-letters. Thus, $\Lm{m} = \emptyset$ if $m < d$. 

It is useful to have a more explicit description of the elements in $\Lm{m}$. To this end, we say $w \in \Sigma^*$ is a \emph{simple word} if it contains no $\tau$-letters, that is, $w \in \{\xi_1,\ldots,\xi_c \}^*$. 

\begin{lem}
\label{lem:desc standard words}
If $m \ge 1$ then the elements in $\Lm{m}$ are precisely the words of the form 
\[
w = w_1 \tau_{i_1} w_2 \tau_{i_2} \ldots w_m \tau_{i_m}, 
\]
where every $w_j$ is  a  standard simple word and $1 \le i_1 \le \cdots \le i_l \le d$ and $i_{l+1} = \cdots = i_m = 0$ for some integer $l$ with $d \le l \le m$ and $[d] \subset \{i_1,,\ldots,i_m\}$. 

Moreover, $\Lm{0}$ is empty if $d \ge 1$ and consists precisely of the empty word if $d = 0$. 
\end{lem}

\begin{proof}
This follows for the most part from the definition of $\Lm{m}$, except for the possibility that a word could end with a standard simple word other than the empty word. 

Consider $w = w_1 \tau_{i_1} w_2 \tau_{i_2} \ldots w_m \tau_{i_m} w_{m+1} \in \Lm{m}$. We have to show that $w_{m+1}$ must be the empty word. Indeed, since $[d] \subset \{i_1,,\ldots,i_m\}$. we get $w_{m+1} \in \overline{\cL_0}$ and that $w_{m+1}$ does not contain the letter $\tau_0$. This implies that $w_{m+1}$ is the empty word. 
\end{proof}

Note that the shortest word in $\Ls$ is $\tau_1 \tau_2 \ldots \tau_d$ and that $\eta (\tau_1 \tau_2 \ldots \tau_d) = (1,2,\ldots,d)$. This is an example of the following observation. 

\begin{lem}
    \label{lem:map mu}
The map $\eta\colon \Sigma^* \to \Mon (\bP) \times \N_0^d$ (see \eqref{eq:def nu} above) induces a map \\
$\tilde{\mu}\colon \Ls \to \Mon (\bP) \times \Hom_{\OI} ([d], \N)$,  
defined by $\tilde{\mu} (w) = \eta (w)$. 
\end{lem}

\begin{proof}
Every shifting operator $T_i$ transforms a $d$-tuple $(p_1,\ldots,p_d)$ with weakly increasing entries $p_1 \le p_2 \le \cdots \le p_d$ into another $d$-tuple with this property. Since every word $w \in \Ls$ contains each of the letters $\tau_1, \tau_2, \ldots, \tau_m$ at least once, it follows that 
$\eta (w) = x^u (p_1,\ldots,p_d)$ with $p_1 < \cdots < p_d$, as claimed. 
\end{proof}

The above map $\tilde{\mu}$ produces the desired bijections. We use multi-index notation $\xi^a$ with $a = (a_1,\ldots,a_c) \in \N_0^c$ to denote the standard simple word $\xi_1^{a_1} \ldots \xi_c^{a_c}$ in which the letter $\xi_i$ is repeated $a_i$ times. 

\begin{prop}
    \label{prop:bijection}
For every integer $m \ge 0$, the map $\tilde{\mu}\colon \Ls \to \Mon (\bP) \times \Hom_{\OI} ([d], \N)$ induces a bijection $\mu_m \colon \Lm{m} \to \Mon (\Fb_m), \ w \mapsto x^u e_{\pi}$,  where $\pi \colon [d] \to [m]$ is the map such that $\tilde{\mu} (w) = x^u (\pi (1),\ldots, \pi (d))$. In particular, this gives a bijection 
\[
\mu \colon \Ls \to \Mon (\Fb) = \coprod_{m \ge 0} \Mon (\bP_m) \times \Hom_{\OI} ([d], [m]). 
\]

More precisely, if $x^u e_{\pi} = x_{\lpnt, 1}^{u_1} \cdots x_{\lpnt, m}^{u_m} e_{\pi} \in \Mon (\Fb_m)$, 
then the unique word $w \in \Lm{m}$ with $\mu (w) = x^u e_{\pi}$ is 
\[
w = \xi^{u_1} \tau_{\iota (1)} \xi^{u_2} \tau_{\iota (2)} \cdots \xi^{u_m} \tau_{\iota (m)},  
\]
where 
\[
\iota (k) = \begin{cases}
0 & \text{ if }  \pi (d) < k \le m\\
j & \text{ if } \pi(j-1) < k \le \pi (j)
\end{cases}
\]
and we set $\pi (0) = 0$. 
\end{prop}

\begin{proof}
Surjectivity of $\mu$ is a consequence of the stated description of a preimage $w$ of $x^u \pi$. Indeed, the definition of $\iota$ gives that  the letter $\tau_1$ occurs $\pi(1)$ times in $w$, the letter $\tau_i$ with $2 \le i \le d$ occurs $\pi (i) - \pi (i-1)$ times, and the letter $\tau_0$ occurs $m - \pi (d)$ times. This implies $\mu (w) = x^u  e_{\pi}$. 

Now we show that $\mu$ is  injective. If $m < d$ the domain and codomain are empty. Let $m \ge d$ and consider words $v, w \in \Lm{m}$ with $\mu (v) = \mu (w) = x^u e_{\pi}$. By definition of $\mu$, it follows that $v$ and $w$ have the same distribution of $\tau$-letters. Thus, \Cref{lem:desc standard words} gives 
\begin{align*}
v & = v_1 \tau_{i_1} v_2 \tau_{i_2} \ldots v_m \tau_{i_m}, \\
w & = w_1 \tau_{i_1} w_2 \tau_{i_2} \ldots w_m \tau_{i_m}
\end{align*}
with standard simple words $v_j, w_j$. Assume $ v \neq w$ and set $k = \min \{ j \in [m] \; | \; v_j \neq w_j\}$. Write $v_k = \xi_1^{a_1} \cdots \xi_c^{a_c}$ and $w_k = \xi_1^{b_1} \cdots \xi_c^{b_c}$. Then one gets 
\begin{align*}
\mu (v) & = f \cdot x_{1,k}^{a_1} \cdots x_{c, k} ^{a_c} \cdot g e_{\pi}\\
\mu (w) & = f \cdot x_{1,k}^{b_1} \cdots x_{c, k} ^{b_c} \cdot h e_{\pi}
\end{align*}
with monomials $f \in \bP_{k-1}$ and $g, h \in K[x_{i, j} \; | \; i \in [c], \ k < j \le m]$. By the choice of $k$, we have $a \neq b$.  It follows   $x_{\lpnt, k}^a \neq x_{\lpnt, k}^b$, and so $\mu (v) \neq \mu (w)$. This contradiction proves $v = w$.   
\end{proof}

For a subset $S$ of $\bF$, we denote by $\la S \ra$ the submodule of $\bF$ generated by $S$. Thus, if $x^u e_{\pi}$ is a monomial in $\bF$, then $\Mon (\la x^u e_{\pi} \ra)$ consists of the monomials $x^v e_{\sigma}$ that are  \emph{$\OI$-divisible} by $x^u e_{\pi}$ (see \cite[Definition 6.1]{NR2}). In particular, one has for a monomial $x_{\lpnt, 1}^{v_1} \cdots x_{\lpnt, n}^{v_n}e_{\rho}$, 
\begin{equation}
\label{eq:def-divisible}
\begin{split}
x_{\lpnt, 1}^{v_1} \cdots x_{\lpnt, n}^{v_n}e_{\rho} \in \la x^u e_{\pi} \ra & \text{ if and only if there is some } \eps \in \Hom_{\FIO} ([m], [n]) \\
& \; \; \; \text{such that } \rho = \eps \circ \pi \text{ and } u_i \le v_{\eps (i)} \text{ for each } i \in [m].
\end{split}
\end{equation}
If $x^{u_1}e_{\pi_1},\ldots,x^{u_s}e_{\pi_s}$ is a set of monomials, observe that 
\begin{align}
     \label{eq:union monomials}
\Mon (\la x^{u_1}e_{\pi_1},\ldots,x^{u_s}e_{\pi_s} \ra ) = \bigcup_{i=1}^s \Mon (\la x^{u_i} e_{\pi_i} \ra). 
\end{align}

Recall that we identify $\Mon (\bF_m)$ with $\Mon (\bP_m) \times \Hom_{\OI} ([d], [m])$ and $\eta_1 \colon \cL \to \Mon (\Pb)$ is defined by Equation \eqref{eq:def nu}.  

\begin{prop}
    \label{prop:inverse of monomial} 
Consider any standard word $w = w_1 \tau_{i_1} w_2 \tau_{i_2} \ldots w_m \tau_{i_m} \in \Lm{m}$ as described in \Cref{lem:desc standard words}. Then the preimage under $\mu$ of the set of monomials in $\la \mu (w) \ra$ is 
\begin{align*}
\mu^{-1} (\Mon (\la \mu (w) \ra) ) = \overline{\cL_{i_1}}  \cN_{w_1} \tau_{i_1} \overline{\cL_{i_2}} \cN_{w_2} \tau_{i_2} \ldots \overline{\cL_{i_m}}  \cN_{w_m} \tau_{i_m }\overline{\cL_0}  \cap \Ls, 
\end{align*}
where $\cN_{w_j}$ is the set of simple words $v$ such that the monomial $\eta_1 (v)$ is divisible by $\eta_1 (w_j)$. 
\end{prop}

\begin{proof}
Observe that  $\cL_{i_1}  w_1 \tau_{i_1} \cL_{i_2} w_2 \tau_{i_2} \ldots \cL_{i_m}  w_m \tau_{i_m} \overline{\cL_0}$ is a subset of $\cL$ and that it consists precisely of the words obtained 
from $w$ by  inserting in front of a substring $w_k \tau_{i_k}$ any string consisting only of letters 
$\tau_{i_k}, \xi_1,\ldots,\xi_c$ or by appending $w$ by a word in $\overline{\cL_0}$. 
 Any such resulting word is clearly in $\cL$. Reordering then neighboring $\xi$-letters suitably converts it into a standard word in $\Pc = \overline{\cL_{i_1}}  \cN_{w_1} \tau_{i_1} \overline{\cL_{i_2}} \cN_{w_2} \tau_{i_2} \ldots \overline{\cL_{i_m}}  \cN_{w_m} \tau_{i_m }\overline{\cL_0}$.
\smallskip 
 
 We first show $\mu^{-1} (\Mon (\la \mu (w) \ra)) \subset \Pc \cap \Ls$. 
 
 To this end set $x^u e_{\pi} = \mu (w)$. The monomials in $\la x^u e_{\pi} \ra$ are of the form $\eps_* (x^u) \cdot x^v \cdot e_{\eps \circ \pi} \in \bF_n$ for some $\eps \in \Hom_{\OI} ([m], [n])$ and some monomial $x^v \in \bP_n$ (see \Cref{rem:poly-alg}(ii) for the definition of $\eps_*$).    We have to show $\mu^{-1} (\eps_* (x^u) \cdot x^v \cdot e_{\eps \circ \pi}) \in \cP \cap \Ls$. We proceed in two steps. 
 
(I)  We begin by proving $\mu^{-1} (\eps_* (x^u) \cdot e_{\eps \circ \pi}) \in \cP \cap \Ls$. 
 
 If $n =m$ then $\eps$ is the identity, and we are done. Assume $n = m+1$, and so $\eps \in \Hom_{\OI} ([m], [m+1])$. Note that $\Hom_{\OI} ([m], [m+1]) = \{\sigma_0,\ldots,\sigma_m \}$, where $\sigma_k$ is defined by 
 \[
 \sigma_k (j) = \begin{cases}
 j & \text{ if } j \le k \\
 j+1 & \text{ if } j > k. 
 \end{cases}
 \]
 
For $k \in [m+1]$, we claim $\mu^{-1} ((\sigma_{k-1})_* (x^u) \cdot e_{\sigma_{k-1} \circ \pi}) = w'$, 
 where $w' = w \tau_0$ if $k = m+1$ and $w'$ is obtained from $w$ by replacing the substring 
 $w_k \tau_{i_k}$ by $\tau_{i_k} w_k \tau_{i_k}$ if $1 \le k \le m$. Indeed, $w'$ is a standard word in 
$\cL$. Thus, it suffices to check $\mu (w') = (\sigma_{k-1})_* (x^u) \cdot e_{\sigma_{k-1} \circ \pi}$. 
This is a routine computation. We omit the details. We have shown that 
$\mu^{-1} (\eps_* (x^u) \cdot e_{\eps \circ \pi}) \in \cP \cap \Ls$ for any 
$\eps \in \Hom_{\OI} ([m], [m+1])$. 

Consider any $\eps \in \Hom_{\OI} ([m], [n])$. with $n > m+1$.  
By \cite[Lemma 2.3]{NR2}, one has 
 \[
 \Hom_{\FIO} ([m], [n]) = \Hom_{\FIO} ([m+1], [n]) \circ \Hom_{\FIO} ([m], [m+1])
 \]
for any positive integers $m < n$. Hence, applying $n-m+1$ times the argument for $n = m+1$ gives $\mu^{-1} (\eps_* (x^u) \cdot e_{\eps \circ \pi}) \in \cP \cap \Ls$. 

(II) It remains to show $\mu^{-1} (x^v \cdot \eps_* (x^u) \cdot e_{\eps \circ \pi}) \in \Pc \cap \Ls$ for every monomial $x^v \in \Pb_n$. 

Indeed, write $\tilde{w} = \mu^{-1} (\eps_* (x^u) \cdot e_{\eps \circ \pi})$ as 
\[
\tilde{w} = \tilde{w}_1 \tau_{j_1} \tilde{w}_2 \tau_{j_2} \ldots \tilde{w}_n \tau_{j_n}
\]
and 
let $w'$ be the word obtained from $\tilde{w}$ by inserting in front of each substring $\tilde{w}_k$ of $\tilde{w}$ the 
string $\xi^{v_k}$. It follows that $w'$ is in $\cP$ as $\tilde{w} \in \cP$ by Step (I). 
 One computes $\eta (w') = x^{\tilde{u}+v} e_{\eps \circ \pi} = x^v \mu (\tilde{w})$, where $x^{\tilde{u}} = \eps_* (x^u)$. Reorder neighboring $\xi$-letters in $w'$ to obtain a standard word $w''$. Then $\mu (w'') = \eta (w'') = \eta (w') = x^v \mu (\tilde{w})$, as desired. 
\smallskip

Secondly, to complete the argument we prove $\mu^{-1} (\Mon (\la \mu (w) \ra) \supset \Pc \cap \Ls$.
Consider any word $v = v_1  w_1 \tau_{i_1} v_2 w_2 \tau_{i_2} \ldots v_m w_m \tau_{i_m} v_0$ in $\cP_n$, that is, $v_0 \in  \overline{\cL_0}$ and $v_k \in \cL_{i_k}$ for $k = 1,\ldots,m$. 

Assume first that $v_0$ is the empty word. 
For $j \in [m]$, denote by $q_j \ge 1$ the number of $\tau$-letters occurring in the substring $v_j w_j \tau_{i_j}$. 
In $v$ replace each $v_j$ by $ \tau_{i_j}^{q_j - 1}$ to obtain the word
\[
v' = \tau_{i_1}^{q_1 - 1}  w_1 \tau_{i_1}  \tau_{i_2}^{q_2 - 1}  w_2 \tau_{i_2} \ldots  \tau_{i_m}^{q_m - 1}  w_m \tau_{i_m}.
\]
It follows that $\eta (v) = x^b \eta (v')$ for some monomial $x^b \in \bP_n$. We claim that there is some $\eps \in \Hom_{\OI} ([m], [n])$ such that  $\eta(v') = \eps (x_u) e_{\eps \circ \pi}$.  This implies $\eta (v) \in \la x^u e_{\pi} \ra$, and thus proves the desired inclusion if $v_0$ is the empty word. 

Indeed, a computation gives 
\[
\eta (v') = x_{\lpnt, q_1}^{u_1} x_{\lpnt, q_1 + q_2}^{u_2} \cdots x_{\lpnt, q_1 + \cdots + q_m}^{u_m} (\rho (1),\ldots,\rho(d)) , 
\]
where $\rho \colon [d] \to [n]$ is the map such that $\rho (k)$ is equal to the number of $\tau$-letters occurring in $v'$ whose index is positive and at most $k$. Since the integers $q_j$ are positive, setting $\eps (j) = q_1 + q_2 + \cdots + q_j$ gives a map $\eps \in \Hom_{\OI} ([m], [n])$ with $\eta_1 (v') = \eps (x^u)$. It remains to show that $\rho = \eps \circ \pi$. 

Note that $p_k = \pi (k)$ is the number of $\tau$-letters occurring in $w$ whose index is positive and at most $k$. By the ordering of the $\tau$-letters in $\cL$, this is equivalent to $p_k = \max \{j \in [m] \; | \; 1 \le i_j \le k \}$. It follows that $\eps (p_k) = q_1 + q_2 + \cdots + q_{p_k}$ is equal to the sum over integers $q_j$ with $1 \le i_j \le k$. The latter is precisely $\rho (k)$. Thus, we have shown $\rho = \eps \circ \pi$, which completes the argument in this case. 

Second, it remains to consider any word $v v_0 \in \cP_n \cap \Ls$ with 
$v_0 \in \overline{\cL_0}$ and $v$ as above, that is, 
$v \in  \overline{\cL_{i_1}}  \cN_{w_1} \tau_{i_1} \overline{\cL_{i_2}} \cN_{w_2} \tau_{i_2} \ldots \overline{\cL_{i_m}}  \cN_{w_m} \tau_{i_m }$. Denote 
by $q_0$ the number of occurrences of $\tau_0$ in $v_0$. Then $\tilde{\mu} (v \tau_0^{q_0}) = \tilde{\mu} (v)$, but $\mu (v \tau_0^{q_0}) \in \Fb_n$  and $\mu (v) \in \bF_{n-q_0}$. 
By the above argument we 
know $\mu (v) \in \la x^u e_{\pi} \ra$ and so $\mu (v \tau_0^{q_0}) \in \la x^u e_{\pi} \ra$. Since $\mu (v v_0) = x^b \mu (v \tau_0^{q_0})$ for some 
monomial $x^b \in \bP_n$, we obtain $\mu (v v_0) \in \la x^u e_{\pi} \ra$, as desired. 
\end{proof}

Let us record a special case of the previous result.  

\begin{cor}
   \label{cor:insertion}
Consider any monomial $x^a e_{\pi} \in \bF_m$ and its corresponding standard word 
\[
\mu^{-1} (x^a e_{\pi}) = w_1 \tau_{i_1} w_2 \tau_{i_2} \ldots w_m \tau_{i_m} \in \Lm{m}.  
\]
Then a monomial $x^b e_{\rho} \in \bF_n$ is equal to $\eps (x^a e_{\pi})$ for some $\eps \in \Hom_{\OI}([m], [n])$ if and only if 
\[
\mu^{-1} (x^b e_{\rho}) \in \{\tau_{i_1}\}^* w_1 \tau_{i_1} \{\tau_{i_2}\}^* w_2 \tau_{i_2} \ldots  \{\tau_{i_m}\}^* w_m \tau_{i_m} \{\tau_{0}\}^*
\]
and $\mu^{-1} (x^b e_{\rho})$ contains exactly $n$ $\tau$-letters. 
\end{cor}

\begin{proof}
This follows from \Cref{prop:inverse of monomial}  because the words $\mu^{-1} (x^a e_{\pi})$ and $\mu^{-1} ( \eps_* (x^a e_{\pi}))$ contain the same number of $\xi$-letters. 
\end{proof}

\begin{rem}
    \label{rem:insertion}
\Cref{prop:inverse of monomial} shows for monomials $x^a e_{\pi}$ and $x^b e_{\rho}$ with $x^b e_{\rho} \in \langle x^a e_{\pi} \rangle$ that the word $v = \mu^{-1} (x^b e_{\rho})$ can be obtained from $w = \mu^{-1} (x^a e_{\pi} )$ by inserting letters suitably. In fact, since $x^b e_{\rho} = x^p \eps_* (x^a e_{\pi})$ by assumption, one can obtain $v$ from $w$ by inserting first only $\tau$-letters  as described in \Cref{cor:insertion} and then inserting only $\xi$-letters suitably. 
\end{rem}

We now want to show that the language described in \Cref{prop:inverse of monomial} is regular. 
Recall that the class of \emph{regular languages} on $\Sigma$ is the smallest class of languages that contains the languages having a letter of $\Sigma$ or the empty word as their only  word and that is closed under taking unions, concatenation and passing from a language $\cN$ to its Kleene star $\cN^*$. The class of regular languages is also closed under intersections and taking complements (see \cite[Section 4.2]{HU}). 

\begin{prop}
     \label{prop:regular language}
For any monomial $x^u e_{\pi} \in \Fb$, the set $\mu^{-1} (\Mon(\langle x^u e_{\pi} \rangle))$ is a regular language on $\Sigma$. 
\end{prop}

\begin{proof}
Using \Cref{prop:bijection} and \Cref{prop:inverse of monomial} and their notation  we know that there is a word $w = w_1 \tau_{i_1} w_2 \tau_{i_2} \ldots w_m \tau_{i_m} \in \Ls$ such that 
\begin{align*}
\mu^{-1} (\Mon (\la x^u e_{\pi}  \ra) ) = \overline{\cL_{i_1}}  \cN_{w_1} \tau_{i_1} \overline{\cL_{i_2}} \cN_{w_2} \tau_{i_2} \ldots \overline{\cL_{i_m}}  \cN_{w_m} \tau_{i_m }\overline{\cL_0}  \cap \Ls. 
\end{align*}
Note that the language $\cQ$ of standard simple words is $\cQ = \{\xi_1\}^* \{\xi_2\}^* \ldots \{\xi_c\}^*$.  Thus, it is regular. For any $i \in [c]$, the identity $\cL_i \cap \Sigma^*_{\rm std} = \cQ (\tau_i \cQ)^*$ shows that $\cL_i \cap \Sigma^*_{\rm std}$ also is a regular language. Similarly, $\overline{\cL_{0}} \cap \Sigma^*_{\rm std}= (\cQ \tau_0)^*$ is regular. Recall that $\cL = \cL_1 \tau_1 \cL_2 \tau_2 \ldots \cL_d \tau_d \overline{\cL_0}$. It follows 
\[
\Ls = \cL \cap \Sigma^*_{\rm std} = (\cL_1 \cap \Sigma^*_{\rm std}) \tau_1 (\cL_2 \cap \Sigma^*_{\rm std})  \tau_2 \ldots (\cL_d \cap \Sigma^*_{\rm std}) \tau_d (\overline{\cL_0} \cap \Sigma^*_{\rm std}). 
\]
Therefore $\Ls$ is a regular language. 

It is not too difficult to check that each of the languages $\cN_{w_i}$ is regular. Hence, $\cP = \overline{\cL_{i_1}}  \cN_{w_1} \tau_{i_1} \overline{\cL_{i_2}} \cN_{w_2} \tau_{i_2} \ldots \overline{\cL_{i_m}}  \cN_{w_m} \tau_{i_m }\overline{\cL_0}$ is a regular language, and so is
$\mu^{-1} (\Mon (\la x^u e_{\pi}  \ra) )  = \cP \cap \Ls$. 
\end{proof}

\begin{cor}
     \label{cor:reg language} 
If $x^{u_1}e_{\pi_1},\ldots,x^{u_s}e_{\pi_s}$ is any finite set of monomials in $\Fb$, then the language \\
 $\mu^{-1} (\Mon (\langle x^{u_1}e_{\pi_1},\ldots,x^{u_s}e_{\pi_s} \rangle))$ is regular. 
\end{cor}

\begin{proof}
Use Indentity \eqref{eq:union monomials}. 
\end{proof}

In order to relate this result to Hilbert series we assume now that $K$ is any field. 

\begin{thm}
     \label{thm:rational hilb monomial submod}
The equivariant Hilbert series of any monomial submodule of $\bF = \Fo{d}$ over $\Pb = (\XO{1})^{\otimes c}$ is a rational function. 
\end{thm}

\begin{proof}
Let $\M$ be a monomial $\OI$-submodule of $\Fb$. By \cite[Theorem 6.15]{NR2}, $\M$ is finitely generated. In particular, there is a finite set of monomials that generates $\M$. Hence \Cref{cor:reg language} gives that $\cN = \mu^{-1} (\Mon (\M))$ is a regular language on $\Sigma$. 

Consider a polynomial ring $T = K[s, t]$ in two variables. Define  a monoid homomorphism 
\begin{align}
    \label{eq:weight function}
\rho \colon \Sigma^{\star} \rightarrow \Mon (T) & \text{ by } \rho(\xi_i) = t \text{ and } \rho (\tau_j) = s  
\end{align}
for any $i \in [c]$ and $j \in\{0,\ldots,d\}$. The generating function of $\cN$ with respect to $\rho$ is a formal power series 
\[
P_{\cN, \rho} (s, t) = \sum_{w \in \cN} \rho (w) =  \sum_{n \ge 0} \, \sum_{w \in \cN_n} \rho (w).  
\]
Since $\cN$ is a regular language a standard result gives that $P_{\cN, \rho} (s, t)$ is a rational function (see, e.g., \cite{H} or \cite[Theorem 4.7.2]{St}).

By definition of $\rho$ one has, for any $w \in \Sigma^*$, that $\rho (w) = s^n t^j$ if $n$ is the number of 
$\tau$-letters occurring in $w$ and $j$ is the number of $\xi$-letters in $w$. 
Since $\M$ is generated by monomials, for any integers $n, j$, the $K$-vector space $[\M_n]_j$ has a basis consisting of all degree $j$ monomials in $\M_n$. Hence \Cref{prop:bijection} shows that $\dim_K [\M_n]_j$  is equal to the number of words $w \in \cN_n$ with $\rho (w) = s^n t^j$. 
It follows for the equivariant Hilbert series of $\M$, 
\[
H_{\M} (s, t) = \sum_{n \ge 0, j \in \Z} \dim_K [\M_n]_j s^n t^j = \sum_{n \ge 0} \, \sum_{w \in \cN_n} \rho (w) = P_{\cN, \rho} (s, t). 
\]
We are done since $P_{\cN, \rho} (s, t)$ is a rational function. 
\end{proof}


\section{Denominators of Hilbert series} 
\label{sec:denominator} 

The main result of the previous section shows that the Hilbert series of a monomial submodule of $\Fo{d}$ is rational. The goal of this section is to derive information on the irreducible factors of the denominator polynomial of such a Hilbert series when it is in reduced form. In particular, it turns out that these factors have at most degree one as polynomials in $s$. This fact has important consequences. However, the proof is much more complicated than the argument in the previous section. A reader willing to accept the main result of this section may skip its other parts. 

We continue to use the previously introduced notation. In particular, we consider monomial submodules of graded free $\OI$-modules $\Fo{d}$ over $\Pb = (\XO{1})^{\otimes c}$, where $c$ is a fixed positive integer and the generator of $\Fo{d}$ has degree zero. We will use induction on $d \ge 0$. If $d = 0$ then most of the desired result has been established in \cite{NR}. The arguments developed in that paper will be of  importance here as well. In addition, we need a decomposition result for certain $\OI$-modules. We begin by establishing this decomposition. 

Throughout this section we assume that $K$ is an arbitrary field. 
Note that any monomial submodule $\M$ of $\bF = \Fo{d}$ has a unique minimal generating set consisting of monomials only. Its elements are called the minimal monomial generators of $\M$. We say that $\M$ is \emph{generated in width $m$} if the width of any minimal monomial generator of $\M$ is at most $m$. 
For any finitely generated monomial module $T$ over 
 some ring $\Pb_n$, we denote by $e^+ (T)$ the largest degree of a 
minimal monomial generator of $T$ if $T \neq 0$.  If $T = 0$, we define $e^+ (T) = - \infty$

Fix some $c$-tuple $e = e_1,\ldots,e_c) \in \N_0^c$ and recall that $x_{\lpnt, 1}^e = x_{1,1}^{e_1} \cdots x_{c,1}^{e_c}$. 
For every integer $n \ge d$, the $\Pb_n$-module $\M_n : x_{\lpnt, 1}^e$ decomposes as 
\[
\M_n : x_{\lpnt, 1}^e = \oplus_{\pi} I_{\pi} e_{\pi},
\]
where the sum is taken over all $\pi \in \Hom_{\FIO} ([d], [n])$ and every $I_{\pi}$ is a monomial 
ideal of $\Pb_n$. 
Using these coefficient ideals we get  
\[
e^+ (\M_n : x_{\lpnt, 1}^e) = \max \{e^+ (I_{\pi}) \; \mid \; \pi \in \Hom_{\FIO} ([d], [n]) \}. 
\]
Finally, let $\x_1 \Pb_n$ denote the ideal of $\Pb_n$ generated by the variables $x_{1,1},\ldots,x_{c,1}$. If the ambient ring $\Pb_n$ is understood from context we often simply write $\x_1$.  
Using that $\Pb_{n-1}$ is isomorphic to $\Pb_n/\x_1 \Pb_n$ as a graded $K$-algebra, we will consider $\Fb_n/(\M_n : x_{\lpnt, 1}^e  + \x_1 \Fb_n)$ as a graded $\Pb_{n-1}$-module. 
We are ready to state the announced decomposition result. 

\begin{prop}
     \label{prop:decomposition} 
Let $\M$  be a monomial submodule of $\bF = \Fo{d}$, where $d\ge 0$. Assume that the width of any monomial minimal generator of $\M$ is at most $m$. 
Set $\Gb = \Fo{d-1}$ if $d \ge 1$ and $\Gb = 0$ if $d = 0$. 
Consider any $e = (e_1,\ldots,e_c) \in \N_0^c$. There are monomial submodules $\Qb'$ of $\Gb$ and $\Qb''$ of $\Fb$ generated in width $m-1$ and $m$, respectively, with $\Qb''_{m-1} = 0$ and the following two properties: 
\begin{itemize}
\item[(a)] For every integer $n \ge m+1$, there are isomorphisms of graded $\Pb_{n-1}$-modules
\begin{align}
     \label{eq:decomp}
\Fb_n/(\M_n : x_{\lpnt, 1}^e  + \x_1 \Fb_n) \cong \Gb_{n-1}/\Qb'_{n-1} \oplus \Fb_{n-1}/\Qb''_{n-1}. 
\end{align}

\item[(b)] 
\[
\sum_{k=0}^{e^+ (\M_m)} \dim_K [\Fb_m/\M_m ]_k \ge \sum_{k=0}^{e^+ (\Qb''_m)} \dim_K [\Fb_m/\Qb''_m]_k, 
\]
and equality is true if and only if $\M_n = \Qb''_n$ for every $n \ge m$. 
\end{itemize}
\end{prop} 

The argument below is constructive. It describes how the modules $\Qb'$ and $\Qb''$ are obtained from $\M : x_{\lpnt, 1}^e$. 

\begin{rem} 
(i)
If $e = 0$, i.e., $x_{\lpnt, 1}^e  = 1$, then $\M_n : x_{\lpnt, 1}^e  = \M_n$. In general, it is   \emph{not} true that  the $\OI$-modules $\Fb/(\M + \x_1 \Fb)$ and  
$(\Gb/\Qb')[-1] \oplus (\Fb/\Qb'')[-1]$ are isomorphic, where $(\Fb/\Qb'')[-1]$ denotes the module 
$\Fb/\Qb''$, but with a shift in width. In fact, the $\OI$-module $\M+\x_1 \Fb$ contains monomials of the form $x_{i,2} e_{\pi}$, whereas  $\M_n + \x_1 \Fb_n$ may not contain any such monomial for any $n$. 
However,  \Cref{prop:decomposition}(a) does give a width-wise decomposition as modules over noetherian polynomial rings. 

(ii) If $e \neq 0$, the $\Pb_n$-modules $\M_n : x_{\lpnt, 1}^e$ do not necessarily form the 
width-wise components of any $\OI$-module. Consider for example the case where $c=1$, 
$d = 2$ and $\M$ is the submodule of $\Fo{2}$ generated in width two by the monomial 
$x_1 x_2 (1,2)$. Then $\M_3 : x_1$ is generated as $K[x_1, x_2, x_3]$-module by 
$x_2 (1,2), \ x_3 (1,3), \ x_2 x_3 (2, 3)$. 
The $K[x_1, x_2]$-module $\M_2 : x_1$ is generated by $x_2 (1, 2)$. The $\OI$-module generated by $x_2 (1,2)$ contains in width three $x_3 (2, 3)$. 
But this monomial is not in $\M_3 : x_1$. 

(iii) The $\OI$-modules $\Qb'$ and $\Qb''$ in Part (a) of \Cref{prop:decomposition} depend on the choice of $e \in \N_0^c$. Note that in the inequality in Part (b), the left-hand side is independent of $e$. 
\end{rem}

\begin{proof}[Proof of \Cref{prop:decomposition}] 
Assume first $d \ge 1$. 
We proceed in several steps. First, we define a map $\Res$ on monomials that induces the desired decompositions. Then $\Res$ is used to define the modules $\Qb'$ and $\Qb''$. The bulk of the 
argument and its most technical part is to compare the width $n$ components of these modules via the map $\Res$ with suitable submodules of $\M_n : x_{\lpnt, 1}^e  + \x_1 \Fb_n$ (see Identities \eqref{eq:submodules1}  and \eqref{eq:submodules2}).  Claim (a) then follows. Finally, a further analysis of $\Qb''$ gives Claim (b). 

For any monomial submodule $\Nb$ of a free $\OI$-module,  recall that  $\Mon (\Nb)$ denotes the set of monomials in $\Nb$. We define a map 
\[
\Res\colon \Mon (\Fb) \setminus \bigcup_{n \ge d}  \x_1 \Fb_n \to \Mon (\Fb [-1]) \cup \Mon (\Gb [-1]) 
\]
by considering two cases for a monomial $x^a e_{\pi} \in \Mon (\Fb_n) \setminus \x_1 \Fb_n$ of width $n$. 

Case 1: Assume $\pi (1) \ge 2$. Then define $\Res (x^a e_{\pi})$ as the monomial $x^b e_{\rho}$, 
where $\rho \colon [d] \to [n-1]$ it the map with $\rho (k) = \pi (k) -1$ and
$x^b$ is the monomial obtained from $x^a$ by replacing every variable $x_{i, j}$ dividing $x^a$ by $x_{i, j-1}$. This is well-defined as $x^a$ is not divisible by any variable $x_{i, 1}$ because $x^a \notin \x_1 \Fb_n$ by assumption. Note that $\Res (x^a e_{\pi}) \in \Fb_{n-1}$.

Case 2: Assume $\pi (1) = 1$. Then define $\Res (x^a e_{\pi})$ as the monomial $x^b e_{\rho}$, 
where $\rho \colon [d-1] \to [n-1]$ it the map with $\rho (k) = \pi (k+1) -1$ and
$x^b$ is, as in Case 1,  the monomial obtained from $x^a$ by replacing every variable $x_{i, j}$ dividing $x^a$ by $x_{i, j-1}$. Thus, $\Res (x^a e_{\pi}) \in \Gb_{n-1}$. 

Observe that in both cases $\Res$ maps a monomial of width $n$ onto a monomial of width $n-1$. 

By construction, the map $\Res$ is injective. In fact, $\Res$ is a bijective map because, for every choice of integers $n, j$, there are  as many monomials of degree $j$ in $\bF_n \setminus \x_1 \bF_n$ as there are in $\Fb_{n-1} \oplus \Gb_{n-1}$

For the arguments below it is instructive to describe the map $\Res$ using the bijection between 
$\Mon (\Fb)$ and the regular language $\Ls$ described in \Cref{prop:bijection}. In order to keep track of $d$, let us denote this bijection by $\mu_d$ here. Consider the word $w = \mu_d^{-1} (x^a e_{\pi}) \in \Lm{n}$. Since $x^a \notin \x_1 \Fb_n$, its left-most letter must be 
$\tau_1$. Delete this letter and denote the resulting word by  $\tilde{w} \in \Sigma^*_{n-1}$. Note 
that the condition $\pi (1) \ge 2$ means precisely that the letter $\tau_1$ occurs at least twice in 
$w$. Hence in Case 1 the word $\tilde{w}$ is a standard word in $\Ls$ and $\Res (x^a) = \mu_d (\tilde{w}) \in \Fb_{n-1}$. 
In Case 2, the word $\tilde{w}$ is not in $\cL$ because the letter $\tau_1$ does not occur in it. Let $w'$ be the word obtained from  $\tilde{w}$ replacing every letter $\tau_i \neq \tau_0$  by $\tau_{i-1}$. Thus, $w'$ is a standard  word in the regular language corresponding to $\Mon (\Gb)$ and $\Res (x^a e_{\pi}) = \mu_{d-1} (w') \in \Gb_{n-1}$. 
\smallskip

Now we define the modules $\Qb'$ and $\Qb''$. For every integer $n \ge d$, we write $\Hom_{\FIO} ([d], [n])$ as the disjoint union 
\[
\Hom_{\FIO} ([d], [n]) = H'_n \cup H''_n 
\]
with 
\[
H'_n = \{\pi \in \Hom_{\FIO} ([d], [n]) \; \mid \; \pi (1) = 1\} \, \text{ and } \; H''_n = \{\pi \in \Hom_{\FIO} ([d], [n]) \; \mid \; \pi (1) \ge 2\}. 
\]
This induces a  decomposition of  $\Fb_n$ as $\Pb_n$-module: 
\[
\Fb_n = F'_n \oplus F''_n,  
\]
where 
\[ 
F'_n = \oplus_{\pi \in H'_n} \Pb_n e_{\pi} \, \text{ and } \;  F''_n = \oplus_{\pi \in H''_n} \Pb_n e_{\pi}. 
\]
Notice that there is no non-trivial decomposition of $\Fb$ as $\OI$-module because $\Fb$ is 
generated by one element. Observe that the set $\{\Res (e_{\pi}) \; \mid \; \pi \in H'_n\}$ generates the $\Pb_{n-1}$-module 
$\Gb_{n-1}$ and that $\{\Res (e_{\pi}) \; \mid \; \pi \in H''_n\}$ generates $\Fb_{n-1}$. 

For every $n \ge d$, there is an analogous decomposition of $\M_n : x_{\lpnt, 1}^e = \oplus_{\pi} I_{\pi} e_{\pi}$ as 
\[
\M_n : x_{\lpnt, 1}^e= (M'_n : x_{\lpnt, 1}^e) \oplus (M''_n : x_{\lpnt, 1}^e),  
\]
where 
\[ 
M'_n : x_{\lpnt, 1}^e = \oplus_{\pi \in H'_n}  I_{\pi}  e_{\pi} \, \text{ and } \;  
M''_n : x_{\lpnt, 1}^e = \oplus_{\pi \in H''_n} I_{\pi}  e_{\pi}. 
\]
Thus, $M'_n : x_{\lpnt, 1}^e = (\M_n : x_{\lpnt, 1}^e) \cap F'_n$ and 
$M''_n : x_{\lpnt, 1}^e = (\M_n : x_{\lpnt, 1}^e) \cap F''_n$. 

Consider now
\begin{align}
     \label{eq:decomp M}
\M_n : x_{\lpnt, 1}^e + \x_1 \Fb_n = \oplus_{\pi \in H'_n \cup H''_n} (I_{\pi} + \x_1 \Pb_n) e_{\pi}. 
\end{align}
Every coefficient ideal can be uniquely rewritten as $I_{\pi} + \x_1 \Pb_n = J_{\pi} + \x_1 \Pb_n$, 
where $J_{\pi}$ is a monomial ideal of $\Pb_n$ with the property that none of its minimal generators  
is divisible by any of the variables $x_{1, 1},\ldots,x_{c, 1}$. 

Define $\Qb'$ as the submodule of $\Gb$ generated by 
$\{\Res (x^a e_{\pi}) \; \mid \; x^a e_{\pi} \in M'_m : x_{\lpnt, 1}^e \setminus \x_1 F'_m\} \subset \Gb_{m-1}$ and $\Qb''$ as the submodule of $\Fb$ generated by 
$\{\Res (x^a e_{\pi}) \; \mid \; x^a e_{\pi}  \in M''_{m+1} : x_{\lpnt, 1}^e \setminus \x_1 F''_{m+1} \} \subset \Fb_{m}$. 
Thus, the $\OI$-module $\Qb'$ is generated in width $m-1$, whereas $\Qb''$ is generated in width $m$. 
\smallskip 

Using that the map $\Res$ is bijective, for each integer $n \ge d$, we write $\langle \Res^{-1} (\Qb'_{n-1})\rangle_{\Pb_n}$ for the submodule of $F'_n$ that is generated by the monomials $\Res^{-1} (x^a e_{\rho})$, where $x^a e_{\rho}$ is a monomial in $\Qb'_{n-1}$.
Similarly, we denote by $\langle \Res^{-1} (\Qb''_{n-1})\rangle_{\Pb_n}$  the submodule of $F''_n$ that is generated by the monomials $\Res^{-1} (x^a e_{\rho})$, where $x^a e_{\rho}$ is a monomial in $\Qb''_{n-1}$ 
It follows that none of the monomial minimal generators of $\langle \Res^{-1} (\Qb'_{n-1})\rangle_{\Pb_n}$ and $\langle \Res^{-1} (\Qb''_{n-1})\rangle_{\Pb_n}$ is divisible by any of the variables $x_{1, 1},\ldots,x_{c, 1}$. 

Our next goal is to establish the following equalities:  
\begin{align}
     \label{eq:submodules1} 
     M'_n : x_{\lpnt, 1}^e + \x_1 F'_n  & = \langle \Res^{-1} (\Qb'_{n-1}) \rangle_{\Pb_n}  + \x_1 F'_n \quad \text{ whenever } n \ge m,  \quad  \text{ and} \\ 
      \label{eq:submodules2} 
     M''_n : x_{\lpnt, 1}^e + \x_1 F''_n&  =  \langle \Res^{-1} (\Qb''_{n-1}) \rangle_{\Pb_n} + \x_1 F''_n \quad \text{ whenever }  n \ge m +1. 
\end{align}

To this end we use the fact that every module generated by monomials has a unique minimal 
generating set consisting of monomials only. Thus, it is enough to compare monomial minimal generators in order to show the above identities. 
\smallskip 

We begin by establishing  \eqref{eq:submodules1}. Let $x^a e_{\pi}$ be a minimal generator of 
$\langle \Res^{-1} (\Qb'_{n-1})\rangle_{\Pb_n}$ that is not in $\x_1 F'_n$, where $n \ge m$. This  
means that  the standard word $\mu_d^{-1} (x^a e_{\pi})$ is of the form $\tau_1 w$, where $w$ 
does not contain $\tau_1$,  and that $\Res (x^a e_{\pi})$ is a minimal generator of $\Qb'_{n-1}$. 
Thus, $x^a e_{\pi}$ is in $M'_n : x_{\lpnt, 1}^e$, which gives $x_{\lpnt, 1}^e x^a e_{\pi} \in M'_n$. 
The latter monomial corresponds to a word of the form $w_1 \tau_1 w$, where $w_1$ is a standard simple word. In particular, $w_1$ consists only of $\xi$-letters. 
Note that 
$\Res (x^a e_{\pi})  = \mu_{d-1} (\tilde{w})$, 
where $\tilde{w}$ is the word obtained from $w$ by replacing each letter $\tau_i \neq \tau_0$ in $w$ by $\tau_{i-1}$. 
Since $\Qb'$ is generated in width $m-1$, we get $\Res (x^a e_{\pi}) = x^p \eps(\Res (x^b e_{\sigma}))$ for some monomials $x^p \in \Pb_{n-1}$ and $x^b e_{\sigma} \in M'_m : x_{\lpnt, 1}^e \setminus \x_1 F'_m$ and some $\eps \in \Hom_{\OI} ([m-1], [n-1])$. In fact, $x^p$ must be $1$ as $x^a e_{\pi}$ was chosen as a minimal generator. 
By \Cref{cor:insertion}, this shows that $\tilde{w}$ can be obtained from $\tilde{v} = \mu_{d-1}^{-1} (\Res (x^b e_{\sigma}))$ by inserting suitably $\tau$-letters. Furthermore, applying $\Res^{-1}$ it 
follows that $x^b e_{\sigma}$ 
corresponds to the standard word $\tau_1 v$, where $v$ is obtained from $\tilde{v}$ be renaming 
each letter $\tau_i \neq \tau_0$ in $\tilde{v}$ by $\tau_{i+1}$. Since $\tilde{w}$ can be obtained from $\tilde{v}$ 
by inserting suitably $\tau$-letters analogous insertions transform $\tau_1 v$ to $\tau_1 w$ as well as $w_1 \tau_1 v$ to $w_1 \tau_1 w$. Hence 
\Cref{cor:insertion} gives that $x_{\lpnt, 1}^e x^a e_{\pi}$ is in $\langle x_{\lpnt, 1}^e x^b e_{\sigma} \rangle$. 
We conclude that $x_{\lpnt, 1}^e x^a e_{\pi}$ is in $\M_n$ because $x_{\lpnt, 1}^e x^b e_{\sigma} \in M'_m \subset \M_m$. 
By the choice of $x^a e_{\pi}$, this monomial is in $F'_n$, which implies $x_{\lpnt, 1}^e x^a e_{\pi} \in M'_n$, and so $x^a e_{\pi} \in M'_n  : x_{\lpnt, 1}^e$. Thus, we have shown $M'_n  : x_{\lpnt, 1}^e + \x_1 F'_n   \supset \langle \Res^{-1} (\Qb'_{n-1})\rangle_{\Pb_n}  + \x_1 F'_n$. 

In order to prove the reverse inclusion consider a minimal generator $x^a e_{\pi}$ of 
$M'_n : x_{\lpnt, 1}^e$ that is not in $\x_1 F'_n$. Thus, there is a minimum degree divisor 
$x_{\lpnt, 1}^{e'}$ of $x_{\lpnt, 1}^e$ such that $x_{\lpnt, 1}^{e'} x^a e_{\pi}$ is a minimal generator 
of $M'_n$. 
In particular, we must have $\pi (1) = 1$. Since $x^a e_{\pi}$ is not in $\x_1 F'_n$, the standard 
word $\mu_d^{-1} (x^a e_{\pi})$ is of the 
form $\tau_1 w$, where $w$ does not contain the letter $\tau_1$. Moreover, $x_{\lpnt, 1}^{e'} x^a e_{\pi}$ corresponds to a standard word of the form $w_1 \tau_1 w$, where $w_1$ is a simple word. 
Using that $\M$ is generated in width 
$m$ by assumption and that $x_{\lpnt, 1}^{e'} x^a e_{\pi}$ is a minimal generator of $M'_n$, we get 
$x_{\lpnt, 1}^{e'} x^a e_{\pi} = \eps (x^{b'} e_{\sigma})$ for some monomial $x^{b'} e_{\sigma} \in \M_m$ and some 
$\eps \in \Hom_{\OI} ([m], [n])$. Observing that $1 = \pi (1) = \eps (\sigma (1))$ we conclude that 
$\sigma (1) = \eps (1) = 1$. It follows that $x^{b'} = x_{\lpnt, 1}^{e'} x^b$, where $x^b e_{\sigma}$ is a monomial that is not in  $\x_1 F'_m$. 
Hence $\mu_d^{-1} (x^b e_{\sigma}) = \tau_1 v$, where the word $v$ does not 
contain the letter $\tau_1$, and $x^b e_{\sigma} \in M'_m : x_{\lpnt, 1}^e$. 
By definition of $\Qb'$, the monomial $\Res (x^b e_{\sigma})$ is in $\Qb'_{m-1}$. 
Furthermore, it corresponds to the standard word $\tilde{v}$ that is obtained from $v$ by replacing 
each letter $\tau_i \neq \tau_0$ by $\tau_{i-1}$. Carrying out these replacements on the $\tau$-letters in $w$ transforms $w$ to a  standard 
word $\tilde{w}$ corresponding to 
$\Res (x^a e_{\pi})$. Since $x^a e_{\pi} = \eps (x^b e_{\sigma})$,  
\Cref{cor:insertion} shows that $\tau_1 w$ can be obtained from $\tau_1 v$ by inserting suitably 
letters drawn from $\{\tau_2,\ldots,\tau_d, \tau_0\}$. Analogous insertions transform $\tilde{v}$ to 
$\tilde{w}$. Using again \Cref{cor:insertion}, we conclude that $\Res (x^a e_{\pi})$ is in 
$\langle \Res (x^b e_{\sigma}) \rangle$. Thus, $\Res (x^b e_{\sigma}) \in \Qb'_{m-1}$ implies 
$\Res (x^a e_{\pi}) \in \Qb'_{n-1}$, which shows 
$x^a e_{\pi} \in \langle \Res^{-1} (\Qb'_{n-1}) \rangle_{\Pb_n}$. This completes the proof of Equality \eqref{eq:submodules1}. 
\smallskip 

The arguments for Identity \eqref{eq:submodules2} are similar, but require an extra step. 
Let $x^a e_{\pi}$ be a minimal generator of 
$\langle \Res^{-1} (\Qb''_{n-1})\rangle_{\Pb_n}$ that is not in $\x_1 F''_n$, where $n \ge m+1$. Thus,  the standard word $\mu_d^{-1} (x^a e_{\pi})$ is of the form $\tau_1 w$, where $w$ 
contains $\tau_1$ at least once and corresponds to  $\Res (x^a e_{\pi}) \in \Qb''_{n-1}$.  
Since $\Qb''$ is generated in width $m$ and $x^a e_{\pi}$ was chosen as a minimal generator, we  
get $\Res (x^a e_{\pi}) = \eps(\Res (x^b e_{\sigma}))$ for some monomial 
$x^b e_{\sigma} \in M''_{m+1} : x_{\lpnt, 1}^e \setminus \x_1 F'_{m+1}$ and some $\eps \in \Hom_{\OI} ([m+1], [n])$. 
Using \Cref{cor:insertion} we conclude as above that $x_{\lpnt, 1}^e x^a e_{\pi}$ is in $M''_n$, which proves 
$M''_n : x_{\lpnt, 1}^e + \x_1 F''_n   \supset  \langle \Res^{-1} (\Qb''_{n-1}) \rangle_{\Pb_n} + \x_1 F''_n$. 

In order to show the reverse inclusion consider a minimal generator $x^a e_{\pi}$ of 
$M''_n : x_{\lpnt, 1}^e$ that is 
not in $\x_1 F''_n$. Thus, $\pi (1) \ge 2$ and the standard word $\mu_d^{-1} (x^a e_{\pi})$ is of the 
form $\tau_1 w$, where $w$ 
contains $\tau_1$ at least once and corresponds to  $\Res (x^a e_{\pi})$. 
Let $x_{\lpnt, 1}^{e'}$ be a minimum degree divisor  of $x_{\lpnt, 1}^e$ such that 
$x_{\lpnt, 1}^{e'} x^a e_{\pi}$ is a minimal generator  of $M''_n$. 
As above, we get 
$x_{\lpnt, 1}^{e'}  x^a e_{\pi} = \eps (x^{b'} e_{\sigma})$ for some monomial 
$x^{b'} e_{\sigma} \in \M_m$ and some $\eps \in \Hom_{\OI} ([m], [n])$. 
Any variable $x_{i, 1}$ dividing $\eps (x^{b'})$ must arise as $\eps (x_{i. 1})$. It follows that $x^{b'}$ 
factors as $x^{b'} = x_{\lpnt, 1}^{e'}  x^b$ and that we must have $\eps (1) = 1$ if 
$x_{\lpnt, 1}^{e'} \neq 1$. We conclude that $x^b e_{\sigma} \in \M_m : x_{\lpnt, 1}^{e}$ and
 $x^a e_{\pi} = \eps (x^b e_{\sigma})$. 
We consider two cases. 

\emph{Case 1.} Assume $x^b e_{\sigma} \in \x_1 \Fb_m$. Thus, the left-most letter of 
$v = \mu_d^{-1} (x^b e_{\sigma})$ is a $\xi$-letter. Furthermore, we must have 
$x_{\lpnt, 1}^{e'} =1$. Indeed, we have seen that the alternative forces $\eps (1) = 1$, which implies 
$x^a e_{\pi} = \eps (x^b e_{\sigma}) \in \x_1 \Fb_m$, a contradiction to the choice of $x^a e_{\pi}$.  
Note that $x_{\lpnt, 1}^{e'} = 1$ yields $x^b e_{\sigma} \in \M_m$. 
Since $x^a e_{\pi} = \eps (x^b e_{\sigma})$ 
\Cref{cor:insertion} gives that $\tau_1 w$ can be obtained from $v$ by inserting suitably 
$\tau$-letters. As $v$ begins with a $\xi$-letter, $\tau_1 w$ can also be obtained from $\tau_1 v$ by suitable insertions. The 
same insertions transform $v$ into $w$, which implies 
$\Res (x^a e_{\pi})\in \langle \Res (\mu_d(\tau_1 v)) \rangle$ by \Cref{cor:insertion}. Note that 
$\mu_d (\tau_1 v)$ is in $\M_{m+1}$ because $x^b e_{\sigma} \in \M_m$. Thus, $\mu_d (\tau_1 v)$ is in $M''_{m+1} \subset M''_{m+1} : x_{\lpnt, 1}^{e}$ because $v$ contains the letter $\tau_1$. It follows that 
$\Res (\mu_d(\tau_1 v))$ is in $\Qb''_{m}$, and so $\Res (x^a e_{\pi})$ is in $\Qb''_{n-1}$, which 
shows $x^a e_{\pi} \in \langle \Res^{-1} (\Qb''_{n-1}) \rangle_{\Pb_n}$. 

\emph{Case 2.} Assume $x^b e_{\sigma} \notin \x_1 \Fb_m$. Thus, $\mu_d^{-1} (x^b e_{\sigma})$ 
is of the form  $\tau_1 v$. Moreover, there is a simple word $w_1$ such that 
$x_{\lpnt, 1}^{e'} x^a e_{\pi}$ corresponds to $w_1 \tau_1 w$ and $w_1 \tau_1 v$ corresponds to 
$x_{\lpnt, 1}^{e'} x^b e_{\sigma}$. Using that  $x^a e_{\pi} = \eps (x^b e_{\sigma})$ and 
\Cref{cor:insertion}, we conclude that $\tau_1 w$ can be obtained from $\tau_1 v$ by suitable 
insertions. 
We consider two situations. First, 
assume $\sigma (1) = 1$, that is, $v$ does not contain $\tau_1$. Since $\tau_1 w$ can be obtained from $\tau_1 v$ by suitable insertions and $w$ contains $\tau_1$, we can insert into $v$ the letter $\tau_1$ to obtain $v'$ such that suitable insertions transform $\tau_1 v'$ to $\tau_1 w$. The same insertions transform $v'$ into $w$ and $w_1 \tau_1 v'$ to $w_1 \tau_1 w$. 
Since $v'$ contains $\tau_1$ we get $\mu_d (w_1 \tau_1 v') \in M''_{m+1}$. Hence $\mu_d (\tau_1 v')$ is in $M''_{m+1} : x_{\lpnt, 1}^{e}$, and so $\Res (\mu_d(\tau_1 v'))$ is in $\Qb''_{m}$. 
Now we conclude as in Case 1 that $x^a e_{\pi} \in \langle \Res^{-1} (\Qb''_{n-1}) \rangle_{\Pb_n}$. 
(Note that this step used the assumption  $n \ge m+1$ as $x^b e_{\sigma} = \mu_d (\tau_1 v)$ is in $M'_{m}$, but $\mu_d (\tau_1 v')$ is in $M''_{m+1}$.) 

Second, 
assume $\sigma (1) \ge 2$, that is, $v$  contains the letter  $\tau_1$. Let $v'$ be a standard word obtained from $v$ by inserting one $\tau$-letter in such a way that $\tau_1 v'$ can be transformed to $\tau_1 w$ by further suitable insertions. Since $\mu_d (w_1 \tau_1 v')$ is in $M''_{m+1}$ the above arguments give again $x^a e_{\pi} \in \langle \Res^{-1} (\Qb''_{n-1}) \rangle_{\Pb_n}$, as desired. This completes the proof of Identity \eqref{eq:submodules2}. 
\smallskip 

Note that by construction  $\langle \Res^{-1} (\Qb'_{n-1}) \rangle_{\Pb_n}$ is a module whose monomial minimal generators are not in $\x_1 F'_n$. Consider Equation \eqref{eq:decomp M} and the definition of the ideals $J_{\pi}$ below it.  Comparing with Identities \eqref{eq:submodules1} and \eqref{eq:submodules2}, it follows that 
\[
\Res^{-1} (\Qb'_{n-1}) = \oplus_{\pi \in H'_n} J_{\pi} e_{\pi} \quad \text{ and }  \quad 
\Res^{-1} (\Qb''_{n-1}) = \oplus_{\pi \in H''_n} J_{\pi} e_{\pi}. 
\]
Thus, $\Qb'_{n-1}$ is generated by the monomial $\Res (x^a e_{\pi})$, where $x^a$ is a minimal generator of $J_{\pi}$  and $\pi$ is in $H'_n$. Furthermore, there is an isomorphism 
\[
F'_n/(\langle \Res^{-1} (\Qb'_{n-1}) \rangle_{\Pb_n}  + \x_1 F'_n) \cong  \oplus_{\pi \in H'_n} \big ( \Pb_n/  (J_{\pi} + \x_1 \Pb_n) \big ) e_{\pi}. 
\]
Denote by $\Res (J_{\pi})$ the monomial ideal of $\Pb_{n-1}$ that is generated by the coefficients of $\Res (x^a e_{\pi})$, where $x^a$ is a minimal generator of $J_{\pi}$. Then there is an isomorphism of $\Pb_{n-1}$-modules 
\[
\big ( \Pb_n/(J_{\pi} + \x_1 \Pb_n) \big )  e_{\pi} \cong \big ( \Pb_{n-1}/\Res (J_{\pi}) \big ) \Res (e_{\pi}).
\] 
Taking also into account that $\{\Res (e_{\pi}) \; \mid \; \pi \in H'_n\}$ generates the $\Pb_{n-1}$-module $\Gb_{n-1}$, the map $\Res$ induces  an isomorphism of $\Pb_{n-1}$-modules 
\[
 F'_n/(\langle \Res^{-1} (\Qb'_{n-1}) \rangle_{\Pb_n}  + \x_1 F'_n) \cong \Gb_{n-1}/\Qb'_{n-1}.  
\]
Combined with Identity \eqref{eq:submodules1}, this gives 
\[
F'_n/(M'_n + \x_1 F'_n)   \cong \Gb_{n-1}/\Qb'_{n-1}. 
\]
Similarly, we obtain an isomorphism of $\Pb_{n-1}$-modules 
\[
F''_n/(M''_n + \x_1 F''_n)   \cong \Fb_{n-1}/\Qb''_{n-1}.  
\]
Since $\Fb_n/(\Fb_n + \x_1 \Fb_n) \cong \big [ F'_n/(M'_n + \x_1 F'_n) \big ]\oplus \big [ F''_n/(M''_n + \x_1 F''_n) \big ]$, Assertion (a) follows. 
\smallskip 

It remains to establish Assertion (b). To this end we claim that $\M_m  \subset \Qb''_m$. Indeed, consider any monomial $x^a e_{\pi}$ in $\M_m$ and its corresponding word $w = \mu_d^{-1} (x^a e_{\pi})$. Then $\mu_d (\tau_1 w)$ is a monomial in $M''_{m+1} \setminus \x_1 F''_{m+1}$. Hence $\Res (\mu_d (\tau_1 w)) = \mu_d (w) = x^a e_{\pi}$ is in $\Qb''_m$. Since $\M_m$ is a monomial module, this implies $\M_m \subset \Qb''_m$.  It follows that 
\begin{align}
    \label{eq:inequality} 
\sum_{k=0}^{e^+ (\M_m)} \dim_K [\Fb_m/\M_m]_k \ge \sum_{k=0}^{e^+ (\M_m)} \dim_K [\Fb_m/\Qb''_m]_k.  
\end{align}
Now we use the fact that $\Qb''_m$ is generated by the monomials $\Res (x^b e_{\rho})$, where $x^b e_{\rho}$ is a minimal generator of $M''_{m+1} : x_{\lpnt, 1}^e$ that is not in $\x_1 F''_{m+1}$. Since $\Res (x^b e_{\rho})$ and $x^b e_{\rho}$ have the same degree this implies 
\[
e^+ (\Qb''_m) \le e^+ (M''_{m+1} : x_{\lpnt, 1}^e) \le  e^+ (M''_{m+1}) \le e^+ (\M_{m+1}). 
\]
By assumption $\M$ is generated in width $m$, which gives $e^+ (\M_{m+1}) \le e^+ (\M_{m})$. Together with the previous estimate this shows that  $e^+ (\M_m) \ge e^+ (\Qb''_m)$. Combined with Inequality \eqref{eq:inequality}, we conclude that 
\[
\sum_{k=0}^{e^+ (\M_m)} \dim_K [\Fb_m/\M_m]_k \ge \sum_{k=0}^{e^+ (\Qb''_m)} \dim_K [\Fb_m/\Qb''_m]_k.  
\]
and that equality is true if and only if $\M_m = \Qb''_m$. Since $\M$ and $\Qb''$ are both submodules of $\Fb$ generated in width $m$, the latter is equivalent to $\M_n = \Qb''_n$ for every $n \ge m$. 
\smallskip

Second, assume $d = 0$.  The argument is analogous, but simpler. For example,  in this case 
$\Res\colon \Mon (\Fb) \setminus \bigcup_{n \ge d}  \x_1 \Fb_n \to \Mon (\Fb [-1])$  is given by  mapping $x^a e_{\pi}$ with $\pi \colon \emptyset \to [n]$ onto $x^b e_{\rho}$,  where $\rho \colon \emptyset \to [n-1]$ and  $x^b$ is obtained from $x^a$ by replacing every variable $x_{i, j}$ dividing $x^a$ by $x_{i, j-1}$.
We leave the details to the reader. 
\end{proof}

It will be useful to consider the following invariants of a monomial  $\OI$-submodule. 

\begin{defn}
   \label{def:size}
Let $\M$  be a monomial submodule of $\bF = \Fo{d}$. Denote by $\wi^+ (\M)$  the maximal width  of a monomial minimal generator of $\M$   if $\M \neq 0$. If $\M$ is trivial  
we define $\wi^+ (\M) = - \infty$. Thus, $\wi^+ (\M)$ is the least integer $n$ such that $\M \neq 0$ is generated by elements whose widths are at most $n$. 

Furthermore, if $\M \neq 0$ set 
\[
\si (\M) = \sum_{j = 0}^{e^+ (\M_{wi^+ (\M)})} \dim_K [\Fb_{wi^+ (\M)}/\M_{wi^+ (\M)}]_j. 
\]
We define $\si (\M) = \infty$ if $\M$ is trivial. 
\end{defn}

Intuitively, we think of $\si (\M)$ as a measure for the size of $\M$ relative to $\Fb$. 

\begin{rem}
     \label{rem:size}
Using \Cref{prop:decomposition} and assuming additionally $m = \wi^+ (\M)$, its Part (b) gives 
\[
\si (\M) \ge \si (\Qb'')
\]
because $\Qb''$ is generated in width $m$ and $\Qb''_{m-1} = 0$, which implies $\wi^+ (\Qb'') = m$. 
\end{rem}

One can refine the inequality in \Cref{rem:size}. For $e \in \N_0^c$, write the decomposition in \Cref{prop:decomposition}(a)  as 
\[
\Fb_n/(\M_n : x_{\lpnt, 1}^e  + \x_1 \Fb_n) \cong \Gb_{n-1}/\Qb'(e)_{n-1} \oplus \Fb_{n-1}/\Qb''(e)_{n-1}, 
\]
where $\Qb'(e)$ is a monomial submodule of $\Gb = \Fo{d-1}$ and $\Qb''(e)$ is a monomial 
submodule of $\Fb$. Here we abuse notation to keep track of the exponent of $x_{\lpnt, 1}^e$ on the left-hand side and to avoid superscripts. (Later we will write $\M(s)$ with $s \in \Z$ to indicate that $\M$ is considered with a grading shifted by $s$ (see above \Cref{thm:shape Hilb f.g. OI-module}). This should not cause confusion here.) 

\begin{cor}
     \label{cor:size comparison} 
Adopt the above notation and assumptions of    \Cref{prop:decomposition}. Consider $e, \tilde{e} \in \N_0^c$ with $e \le \tilde{e}$, componentwise. Then one has 
\[
\si (\Qb''(e)) \ge \si (\Qb''(\tilde{e})). 
\]
\end{cor}

\begin{proof}
The assumption $e \le \tilde{e}$ yields $\M : x_{\lpnt, 1}^e \subset \M : x_{\lpnt, 1}^{\tilde{e}}$, which implies $\Qb''(e) \subset \Qb''(\tilde{e})$ and $e^+(\Qb''(\tilde{e})_m) \le e^+ (\Qb''(e)_m)$. The latter inequality is true because both modules are generated by monomials.  Since $m = \wi^+ (\Qb''(e)) = \wi^+ (\Qb''(\tilde{e})$ the assertion follows. 
\end{proof}

As further preparation, we extend some of the methods developed in \cite{NR}. The following observation is similar to \cite[Lemma 6.8]{NR}. 

\begin{lem}
       \label{lem:repeated division}
Fix $n \in \N$ and  let $T$ be a graded submodule of a finitely generated graded free $\Pb_n$-module $F$, and let $\ell \in \Pb_n$ be a linear form such that $T : \ell^r = T : \ell^{r+1}$ for some integer $r \ge 0$. Then one has for the Hilbert series of $F/T$: 
\[
H_{F/T} (t) = \sum_{e = 0}^{r-1} H_{F/(T: \ell^e + \ell F)} (t) \cdot t^e
+ H_{F/( T : \ell^r + \ell F)} (t) \cdot \frac{t^r}{1-t} .
\]
\end{lem}

\begin{proof}
For every integer $j \ge 0$, multiplication by $\ell$ on $F/T : \ell^j$ induces an exact sequence of graded modules 
\[
0 \to \big(F/T : \ell^{j+1} \big ) (-1) \to F/T : \ell^j \to F/(T : \ell^j + \ell F) \to 0. 
\]
Now one concludes as in \cite[Lemma 6.8]{NR}.
\end{proof}

Using this result repeatedly, we obtain the following version. 

\begin{lem}
       \label{lem:gen repeated division}
Let $T$ be a monomial submodule of a finitely generated graded free $\Pb_n$-module $F$ for some $n \in \N$. Fix an integer $r > 0$ such that, for every $i \in [c]$, the monomial $x_{i, 1}^r$ does not divide any monomial minimal generator of $T$. Then the Hilbert series of $F/T$ can be written as  
\[
H_{F/T} (t) = \sum_{e = (e_1,\ldots,e_c)\in \Z^c \atop 0 \le e_l \le r}
\frac{t^{|e|}}{(1-t)^{\gamma (e)}} \cdot H_{F/(T : x_{\lpnt, 1}^e + \x_1 F)}  (t), 
\]
where $|e| = e_1 + \cdots + e_c$ and $\gamma (e) = \# \{e_l \; \mid \; e_l = r \text{ and } 1 \le l \le c\}$. 
\end{lem} 

\begin{proof} Using induction on $k$, we show more generally for $k \in [c]$ that 
\begin{align}
     \label{eq:hilb-inductive}
H_{F/T} (t) = \sum_{e = (e_1,\ldots,e_k)\in \Z^k \atop 0 \le e_l \le r}
\frac{t^{|e|}}{(1-t)^{\gamma_k (e)}} \cdot H_{F/(T : x_{1, 1}^{e_1} \cdots x_{k, 1}^{e_k} + (x_{1, 1}\ldots,x_{k, 1}) F)}  (t), 
\end{align}
where $\gamma_k (e) = \# \{e_l \; \mid \; e_l = r \text{ and } 1 \le l \le k\}$. For $k = c$, this proves the desired statement. 

Let $k = 1$. The assumption that $x_{1,1}^r$ does not divide any monomial generator of $T$ 
means that $T : x_{1,1}^r = T = T :x_{1,1}^{r+1}$. Hence, we may apply \Cref{lem:repeated division}  
and Equation \eqref{eq:hilb-inductive} follows. 

Assume $2 \le k \le c$. Since $T$ is a monomial module, one has 
\[
\big ( T : x_{1, 1}^{e_1} \cdots x_{k-1, 1}^{e_{k-1}} + (x_{1, 1}\ldots,x_{k-1, 1}) F \big ) : x_{k, 1}^{e_k} = T : x_{1, 1}^{e_1} \cdots x_{k-1, 1}^{e_{k-1}} x_{k, 1}^{e_k} + (x_{1, 1}\ldots,x_{k-1, 1}) F 
\]
and 
\begin{align}
     \label{eq:colon modules}
\hspace{2em}&\hspace{-2em}
\big ( T : x_{1, 1}^{e_1} \cdots x_{k-1, 1}^{e_{k-1}} + (x_{1, 1}\ldots,x_{k-1, 1}) F \big ) : x_{k, 1}^{e_k} + x_{k, 1} F
\\
& \hspace{2em} \nonumber   
= T : x_{1, 1}^{e_1} \cdots x_{k-1, 1}^{e_{k-1}} x_{k, 1}^{e_k}   + (x_{1, 1}\ldots,x_{k-1, 1}, x_{k, 1}) F. 
\end{align}
Using that $x_{k, 1}^r$ does not divide any minimal generator of $T$, 
the first equality implies that 
\[
\big ( T : x_{1, 1}^{e_1} \cdots x_{k-1, 1}^{e_{k-1}} + (x_{1, 1}\ldots,x_{k-1, 1}) F \big ) : x_{k, 1}^{r} = 
\big ( T : x_{1, 1}^{e_1} \cdots x_{k-1, 1}^{e_{k-1}} + (x_{1, 1}\ldots,x_{k-1, 1}) F \big ) : x_{k, 1}^{r+1}. 
\]
Hence,   \Cref{lem:repeated division} is applicable to any module $T : x_{1, 1}^{e_1} \cdots x_{k-1, 1}^{e_{k-1}} + (x_{1, 1}\ldots,x_{k-1, 1}) F$ with $0 \le e_l \le r$ by using $\ell = x_{k,1}$. Combined with the induction hypothesis this gives the desired Equation \eqref{eq:hilb-inductive} because of Equality \eqref{eq:colon modules} and $\gamma_k (e_1,\ldots,e_k) = \gamma_{k-1} (e_1,\ldots,e_{k-1}) + \gamma_1 (e_k)$. 
\end{proof}

We are now ready for the main result of this section. In particular, it gives a restriction on  the 
irreducible factors appearing in the denominator of an equivariant Hilbert series. Recall that we 
are considering $\OI$-modules over $\Pb = (\XO{1})^{\otimes c}$. 

\begin{thm}
     \label{thm:shape Hilb mon submodule}
If $\M$ is a monomial submodule of $\bF = \Fo{d}$ with $d \ge 0$, then the equivariant Hilbert series of $\Fb/\M$ is of the form
\[
H_{\Fb/\M} (s, t) = \frac{g(s, t)}{(1-t)^a \cdot \prod_{j =1}^b [(1-t)^{c_j} - s \cdot f_j (t)]},
\]
where $a, b, c_j$ are non-negative integers with $c_j \le c$, \ $g (s, t) \in \Z[s, t]$, and each $f_j (t)$ is a polynomial in $\Z[t]$ satisfying $f_j (1) >  0$ and $f_j (0) = 1$.
\end{thm}

\begin{proof}
Define $\Fo{-1}$ as the zero module over $\Pb$.  We use induction on $d \ge -1$. If $d = -1$ the claim is clearly true. Let $d \ge 0$. If $\M = 0$ then \Cref{prop:hilb free OI-mod} shows the assertion. 

Let $\M$ be non-trivial. Thus $wi^+ (\M)$ and $\si (\M)$ are non-negative integers. Set $m = \wi^+ (\M)$, and so $m \ge d$.  
Recall that, for every $n$, the module $\Fb_n/\M_n$ is graded and finitely generated over the noetherian polynomial ring 
$\Pb_n$. Thus its Hilbert series is of the form 
\[
H_{\Fb_n/\M_n} (t) = \frac{g_n (t)}{(1-t)^{c n}}, 
\]
where $g_n (t)$ is a polynomial in $\Z[t]$. Hence, it suffices to show that the formal power series 
\[
\sum_{n \ge m} H_{\Fb_n/\M_n} (t) s^n = \sum_{n \ge m, j \ge 0} \dim_K [\Fb_n/\M_n]_j s^n t^j
\]
is of the form as stated in the theorem. 
We now use induction on $\si (\M) \ge 0$. If $\si (\M) = \sum_{j=0}^{e^+(\M_m)} \dim_K [\Fb_m/\M_m]_j = 0$, then $\M_m = \Fb_m$, and thus $\Fb_n/\M_n = 0$ for every $n \ge m$ as $m \ge d$. This gives 
$\sum_{n \ge m} H_{\M_n} (t) s^n = 0$, and we are done.  

Let $\si (\M) \ge 1$. 
Since $\M_m$ is finitely generated there is an integer $r > 0$ such that none of the powers $x_{1,1}^r,\ldots,x_{c,1}^r$ divides any 
of the monomial minimal generators of $\M_m$. Thus, for every $n \ge m$, the monomial minimal generators of $\M_n$ are 
also not divisible by any of these powers. Applying \Cref{lem:gen repeated division} to every module $\M_n$ we obtain  
\begin{align*}
\sum_{n \ge m+1} H_{\Fb_n/\M_n} (t) s^n  
& = \sum_{n \ge m+1} \left [ \sum_{e = (e_1,\ldots,e_c)\in \Z^c \atop 0 \le e_l \le r}
\frac{t^{|e|}}{(1-t)^{\gamma (e)}} \cdot H_{\Fb_n/(\M_n : x_{\lpnt, 1}^e + \x_1 \Fb_n)}  (t) \right ] \cdot s^n \\
& = \sum_{e = (e_1,\ldots,e_c)\in \Z^c \atop 0 \le e_l \le r} \frac{t^{|e|}}{(1-t)^{\gamma (e)}} \cdot \left [  \sum_{n \ge m+1} H_{\Fb_n/(\M_n : x_{\lpnt, 1}^e + \x_1 \Fb_n)}  (t)  s^n \right ]. 
\end{align*}
For every $e$ and $n \ge m+1$, \Cref{prop:decomposition}(a) gives a decomposition of the form 
\[
\Fb_n/(\M_n : x_{\lpnt, 1}^e  + \x_1 \Fb_n) \cong \Gb_{n-1}/\Qb'(e)_{n-1} \oplus \Fb_{n-1}/\Qb''(e)_{n-1}, 
\]
where $\Qb'(e)$ is a monomial submodule of $\Gb = \Fo{d-1}$ and $\Qb''(e)$ is a monomial 
submodule of $\Fb$. Here we abuse notation as in \Cref{rem:size}. 
Using these decompositions, we get 
\begin{align*} 
\hspace{2em}&\hspace{-2em}
\sum_{n \ge m+1} H_{\Fb_n/\M_n} (t) s^n \\
& =  \sum_{e = (e_1,\ldots,e_c)\in \Z^c \atop 0 \le e_l \le r} \frac{t^{|e|}}{(1-t)^{\gamma (e)}} \cdot \left [
\sum_{n \ge m+1} H_{\Gb_{n-1}/\Qb'(e)_{n-1}} s^n  +  \sum_{n \ge m+1} H_{\Fb_{n-1}/\Qb''(e)_{n-1}} (t) s^n
\right ] \\
& =  \sum_{e = (e_1,\ldots,e_c)\in \Z^c \atop 0 \le e_l \le r} \frac{t^{|e|}}{(1-t)^{\gamma (e)}} \cdot s   \cdot  \left [
 \sum_{n \ge m} H_{\Gb_{n}/\Qb'(e)_{n}} s^{n} 
\right ] \\
& \hspace*{1cm}  +  \sum_{e = (e_1,\ldots,e_c)\in \Z^c \atop 0 \le e_l \le r} \frac{t^{|e|}}{(1-t)^{\gamma (e)}} \cdot s   \cdot  \left [ 
 \sum_{n \ge m} H_{\Fb_{n}/\Qb''(e)_{n}} (t) s^{n} 
\right ].  
\end{align*} 
By induction on $d$, each of the formal power series 
$\sum_{n \ge m} H_{\Gb_{n}/\Qb'(e)_{n}} s^{n} $ is rational of the desired form. Collecting terms, we  
conclude 
\begin{align*}
\sum_{n \ge m+1} H_{\Fb_n/\M_n} (t) s^n 
& = \frac{h(s, t)}{f(s, t)} + 
\sum_{e = (e_1,\ldots,e_c)\in \Z^c \atop 0 \le e_l \le r} \frac{t^{|e|}}{(1-t)^{\gamma (e)}} \cdot s   \cdot  \left [ 
 \sum_{n \ge m} H_{\Fb_{n}/\Qb''(e)_{n}} (t) s^{n}
\right ],  
\end{align*}
where $\frac{h(s, t)}{f(s, t)}$ is a rational function as described in the right-hand side of the statement. 

By \Cref{prop:decomposition}(b) and \Cref{rem:size}, we know for every $e$ that $\si (\Qb''(e)) \le \si (\M)$. Thus, we can re-write the last equality as 
\begin{align*}
\hspace{2em}&\hspace{-2em}
\sum_{n \ge m} H_{\Fb_n/\M_n} (t) s^n - H_{\Fb_m/\M_m} (t) s^m  -  \frac{h(s, t)}{f(s, t)} \\
& =  
\sum_{e = (e_1,\ldots,e_c)\in \Z^c \atop 0 \le e_l \le r} 
\left [ 
\sum_{\si (\Qb''(e)) < \si (\M)} 
\frac{t^{|e|}}{(1-t)^{\gamma (e)}} \cdot s   \cdot  
  \sum_{n \ge m} H_{\Fb_{n}/\Qb''(e)_{n}} (t) s^{n}\right ] \\
 &  \hspace*{1cm}  + \sum_{e = (e_1,\ldots,e_c)\in \Z^c \atop 0 \le e_l \le r} 
\left [ 
 \sum_{\si (\Qb''(e)) = \si (\M)} \frac{t^{|e|}}{(1-t)^{\gamma (e)}} \cdot s   \cdot  
  \sum_{n \ge m} H_{\Fb_{n}/\Qb''(e)_{n}} (t) s^{n} 
\right ].  
\end{align*} 
By induction on $\si (\M)$, we know that $\sum_{n \ge m} H_{\Fb_{n}/\Qb''(e)_{n}} (t) s^{n} $ 
has the desired form if 
$\si (\Qb''(e)) < \si (\M)$. This implies 
\begin{align*}
\hspace{2em}&\hspace{-2em}
\sum_{n \ge m} H_{\Fb_n/\M_n} (t) s^n - \frac{\tilde{h}(s, t)}{\tilde{f}(s, t)} \\
& =  
 \sum_{e = (e_1,\ldots,e_c)\in \Z^c \atop 0 \le e_l \le r} 
\left [ 
 \sum_{\si (\Qb''(e)) = \si (\M)} \frac{t^{|e|}}{(1-t)^{\gamma (e)}} \cdot s   \cdot  
 \sum_{n \ge m} H_{\Fb_{n}/\Qb''(e)_{n}} (t) s^{n} 
\right ],  
\end{align*} 
where 
$\frac{\tilde{h}(s, t)}{\tilde{f}(s, t)}$ is a rational function as described in the statement. Now we use \Cref{prop:decomposition} again. Its part (b) says $\Qb''(e)_n = \M_n$ for $n \ge m$ if $\si (\Qb''(e)) = \si (\M)$. Substituting we get 
\begin{align*}
\hspace{2em}&\hspace{-2em}
\sum_{n \ge m} H_{\Fb_n/\M_n} (t) s^n - \frac{\tilde{h}(s, t)}{\tilde{f}(s, t)} \\
& =  
 \sum_{e = (e_1,\ldots,e_c)\in \Z^c \atop 0 \le e_l \le r} 
\left [ 
 \sum_{\si (\Qb''(e)) = \si (\M)} \frac{t^{|e|}}{(1-t)^{\gamma (e)}} \cdot s   \cdot  
 \sum_{n \ge m} H_{\Fb_{n}/\M_{n}} (t) s^{n} 
\right ],  
\end{align*}
and so 
\begin{align}
     \label{eq:almost}
\left [ 1 - 
 \sum_{e = (e_1,\ldots,e_c)\in \Z^c, 0 \le e_l \le r  \atop  \si (\Qb''(e)) = \si (\M)} 
 \frac{t^{|e|} \cdot s}{(1-t)^{\gamma (e)}}   \right ]   
 \cdot   \sum_{n \ge m} H_{\Fb_n/\M_n} (t) s^n = \frac{\tilde{h}(s, t)}{\tilde{f}(s, t)}. 
\end{align}

If there is no module $\Qb''(e)$ with $\si (\Qb''(e)) = \si (\M)$ the left-hand is simply $\sum_{n \ge m} H_{\Fb_n/\M_n} (t) s^n$, and  we are done. 

Assume there is some $\Qb''(e)$ with $0 \le e_l \le r$ and $\si (\Qb''(e)) = \si (\M)$. Thus the following number is well-defined:
\[
\gamma = \max \{ \gamma (e) \; \mid \; e \in \Z^c, \ 0 \le e_l \le r, \  \si (\Qb''(e)) = \si (\M) \}.
\]
Observe that $\gamma \le c$ because $\gamma (e) \le c$ for every $e$ by definition (see \Cref{lem:gen repeated division}).  Define a polynomial $p(t) \in \Z[t]$ by 
\begin{equation}
     \label{eq:def of p(t)}
p(t) = \sum_{e = (e_1,\ldots,e_c)\in \Z^c, 0 \le e_l \le r  \atop  \si (\Qb''(e)) = \si (\M)}  
(1-t)^{\gamma - \gamma (e)} \cdot t^{|e|}. 
\end{equation}
Evaluating at $1$, we get
\[
p(1) = \# \{e \; \mid \; e \in \Z^c, \ 0 \le e_l \le r, \  \si (\Qb''(e)) = \si (\M), \ \gamma (e) = \gamma \},  
\]
which is positive by definition of $\gamma$. We claim that $p(0) = 1$. Indeed, we assumed that there is some $\Qb''(e)$ with $\si (\Qb''(e)) = \si (\M)$. Since $e \ge 0 = (0,\ldots,0) \in \N_0^c$ \Cref{cor:size comparison} gives $\si (\Qb''(0)) \ge \si (\Qb''(e))$. Combined with  \Cref{prop:decomposition}(b), we obtain
\[
\si (\M) \ge \si ( \Qb''(0)) \ge \si (\Qb''(e)) = \si (\M), 
\]
that is, $\si (\Qb''(0)) = \si (\M)$. Hence, in the definition of $p(t)$ there is a summand with $e = 0$. It is equal to $(1-t)^{\gamma}$, and so $p (0) = 1$, as claimed.

Finally, Equation \eqref{eq:almost} can be re-written as 
\[
\frac{(1-t)^{\gamma} -  s \cdot p(t)}{(1-t)^{\gamma}} \cdot   \sum_{n \ge m} H_{\Fb_n/\M_n} (t) s^n = \frac{\tilde{h}(s, t)}{\tilde{f}(s, t)}. 
\]
Since we have already shown that $\gamma \le c$ and $p(1) > 0$, $p(0) = 1$, it follows that $\sum_{n \ge m} H_{\Fb_n/\M_n} (t) s^n$ has the desired form. 
\end{proof}

The argument at the end of above proof gives more information about the polynomial $p(t)$ than stated in \Cref{thm:shape Hilb mon submodule}. 

\begin{rem}
    \label{rem:refined numerator for c =1}
Adopt the assumptions and notation of \Cref{thm:shape Hilb mon submodule} and its proof. 
    
(i) 
Assume that the rational function describing the Hilbert series of $\M$ is in reduced form, that is, numerator and denominator are relatively prime in $\Q[s, t]$. 
The proof of \Cref{thm:shape Hilb mon submodule} shows that every irreducible factor of the denominator other than $(1-t)$ and $[(1-t)^c - s]$ is an irreducible factor of some polynomial $[(1-t)^{\gamma} - s \cdot p(t)]$, where $p(t)$ is defined in  Equation  \eqref{eq:def of p(t)}) and $\gamma$ just above it.

(ii)    
Assume there is some $\tilde{e} \in \N_0^c$ with $\si (\Qb''(\tilde{e}))  = \si (\M)$ and consider the  
polynomial 
\[
p(t) = \sum_{e = (e_1,\ldots,e_c)\in \Z^c, 0 \le e_l \le r  \atop  \si (\Qb''(e)) = \si (\M)}  
(1-t)^{\gamma - \gamma (e)} \cdot t^{|e|}. 
\]
Then  \Cref{cor:size comparison} and the above argument show each $e \in \Z^c$ with $0 \le e \le \tilde{e}$ contributes a summand $(1-t)^{\gamma - \gamma (e)} \cdot t^{|e|}$ to $p(t)$. 
\end{rem}

If $c=1$ this observation gives a rather precise description of the denominator of an equivariant Hilbert series. 

\begin{thm}
     \label{thm:shape Hilb mon submodule c = 1}
If $\M$ is a monomial submodule of $\bF = \Fo{d}$ over $\XO{1}$ with $d \ge 0$, then the equivariant Hilbert series of $\Fb/\M$ is of the form
\[
H_{\Fb/\M} (s, t) = \frac{g(s, t)}{(1-t)^a \cdot (1-t-s)^k \prod_{j =1}^b [1 - s \cdot (1 + t + \cdots + t^{e_j})]},
\]
where $a, b, k, e_j$ are non-negative integers and  \ $g (s, t) \in \Z[s, t]$. 
\end{thm}

\begin{proof}
Write $H_{\Fb/\M} (s , t) = \frac{g(s, t)}{f(s, t)}$ with $g(s, t), f(s, t) \in \Z[s, t]$ and $f(s, t)$ as described in   \Cref{thm:shape Hilb mon submodule}. We may assume that none of the irreducible factors of $f(s, t)$ divides the numerator $g(s, t)$. 
Using that $c = 1$, \Cref{rem:refined numerator for c =1}(ii) shows that it remains to consider 
irreducible factors of some polynomial $[(1-t)^{\gamma} - s \cdot p(t)]$, where $p(t)$ is defined in  Equation  \eqref{eq:def of p(t)}) and $\gamma$ just above it. Assume there is an integer with $0 \le e \le r$ and $\Qb''(e)) = \si (\M)$. Let $\tilde{e}$ be the maximum integer $e$ with these properties. Thus,  
\[
p(t) = \sum_{e  \in \Z, 0 \le e \le \tilde{e}  \atop  \si (\Qb''(e)) = \si (\M)}  
(1-t)^{\gamma - \gamma (e)} \cdot t^{|e|}. 
\]
Now we consider two cases. If $\tilde{e} = r$ we get $\gamma (\tilde{e}) = 1$ and $\gamma (e) = 0$ if $0 \le e < \tilde{e}$, and so $\gamma = 1$. If $\tilde{e} < r$ we obtain $\gamma (e) = 0$ if $0 \le e \le \tilde{e}$, and so $\gamma = 0$. 
Using \Cref{rem:refined numerator for c =1}, this yields
\begin{align*}
p(t) & = \sum_{e = 0}^{\tilde{e}}   (1-t)^{\gamma - \gamma (e)} \cdot t^{|e|} \\
& = \begin{cases}
(1-t)[1 + t + \cdots + t^{\tilde{e}-1}] + t^{\tilde{e}} = 1 & \text{ if } \tilde{e} = r; \\
1 + t + \cdots + t^{\tilde{e}} & \text{ if } \tilde{e} < r.   
\end{cases}
\end{align*}
It follows
\begin{align*}
(1-t)^{\gamma} - s \cdot p(t) 
& = \begin{cases}
(1-t)- s  & \text{ if } \tilde{e} = r; \\
1 - s[1 + t + \cdots + t^{\tilde{e}}] & \text{ if } \tilde{e} < r,    
\end{cases}
\end{align*}
which completes the argument. 
\end{proof}

\begin{rem}
      \label{rem:optimality denominator}
The descriptions of the denominator polynomials in \Cref{thm:shape Hilb mon submodule} and   \Cref{thm:shape Hilb mon submodule c = 1} seem rather efficient. 

(i)   If $c = 1$ and $\M$ is the submodule of $\Fo{0}$ that is generated in width $m$ by $x_1^{e_1} \cdots x_{b-1}^{e_{b-1}} x_m^{e_b}$ with $(e_1,\ldots,e_b) \in \N^b$, then \cite[Theorem 1.1]{GN} gives for the equivariant Hilbert series
\[
H_{\Fb/\M} (s, t) = \frac{g(s, t)}{(1-t)^{m-1}  \prod_{j =1}^b [1 - s \cdot (1 + t + \cdots + t^{e_j-1})]}
\]
with some polynomial  $g(s, t) \in \Z[s,t]$ that is not divisible by any of the irreducible factors of the denominator. The Hilbert series of $\Fo{d}$ over $\XO{1}$ is (see \Cref{prop:hilb free OI-mod}) 
\[
H_{\Fo{d}} (s, t) = \frac{(1-t) \cdot s^d }{(1-t-s)^{d+1}}. 
\]
Taking suitable direct sums it follows that any of the denominators permissible by \Cref{thm:shape Hilb mon submodule c = 1} can indeed be realized as the Hilbert series of some $\OI$-module over  $\XO{1}$. 

(ii) Fix any integer $c \ge 1$. Let $0 \neq f (t) \in \Z[t]$ be the numerator polynomial of a Hilbert series of some graded quotient of a polynomial ring in $c$ variables, that is, 
\[
H_A (t) = \frac{f(t)}{(1-t)^D}, 
\]
where $A = K[y_1,\ldots,y_c]/J$ for some homogeneous ideal $J$ and $f(1) > 1$. Then it is known that $D$ is the Krull dimension of $A$ and $f(0) = 1$. By \cite[Example 7.3]{NR}, there is an ideal $\Ib$ of $\Fo{0}$ over $(\XO{1})^{\otimes c}$ with equivariant Hilbert series
\[
H_{\Fo{0}/\Ib} (s, t) = \frac{(1-t)^D}{(1-t)^D - s f(t)}. 
\]
Taking suitable direct sums it follows that any polynomial $[(1-t)^D- s f(t)]$ with $f(t)$ and $D$ as above can appear as an irreducible factor in the denominator in an equivariant Hilbert series when written in reduced form. This indicates the effectiveness of the restrictions on the polynomials $f_j(t)$ given in \Cref{thm:shape Hilb mon submodule} and \Cref{rem:refined numerator for c =1}. 
\end{rem}


\section{Finitely Generated $\FI$- and $\OI$-modules}
\label{sec:general case}

We extend the main result of the previous section from monomial submodules to arbitrary finitely generated graded modules. As a consequence we obtain results on the growth of invariants of the width-wise components of any such $\OI$- or $\FI$-module over a noetherian polynomial $\OI$- or $\FI$-algebra, respectively. 

We first need an extension of some results on Gr\"obner bases in \cite{NR2}. These are true in great generality. In fact, at first we consider $\OI$-modules that are not necessarily graded, Moreover, let $K$ be any commutative noetherian ring. 

Fix integers $c \ge 1$ and $d_1,\ldots,d_k \ge 0$, and consider the $\OI$-module 
\[
\Fb = \bigoplus_{i=1}^k \Fo{d_i}
\]
over $\Pb = (\XO{1})^{\otimes c}$. A \emph{monomial} in $\Fb$ is a monomial of some summand 
$\Fo{d_i}_m$, that is, it is an element  of the form
\[
x^a  e_{\pi, i} = x_{\lpnt, 1}^{a_1} \cdots x_{\lpnt, m}^{a_m} e_{\pi, i}, \quad \text{ where } \pi \in \Hom_{\FIO} ([d_i], [m]) \text{ for some } m \in \N, i \in [k], \; a_j \in \N_0^c.   
\]
The second index in $e_{\pi, i}$ is used in order to distinguish the monomials in the direct summands. We need to order the monomials in $\Fb$. 

\begin{defn}
A \emph{monomial order} on $\Fb$ is a total order $>$ on the monomials of $\Fb$ such that
 if $x^a  e_{\pi, i}, x^b  e_{\rho, j} $ are monomials in $\Fb_m$ , then $x^a  e_{\pi, i} > x^b  e_{\rho, j} $ implies:
\begin{enumerate}
\item $x^p x^a  e_{\pi, i}   > x^p x^b  e_{\rho, j}  > x^b  e_{\rho, j}$ for every monomial $x^p \neq 1$ in $\Pb_m$;

\item $\eps_*(x^a) e_{\eps \circ \pi, i}  > \eps_*(x^b)  e_{\eps \circ \rho, j}$ for every $\eps \in \Hom_{\FIO} ([m], [n])$; \quad and

\item $e_{\iota_{m, n} \circ \pi, i}  > e_{ \pi, i} $ whenever $n > m$, where $\iota_{m, n}\colon [m] \to [n], l \mapsto l$. 
\end{enumerate}

\end{defn}

Such orders exist.

\begin{ex}
\label{exa:mon-order}
Order the monomials in every polynomial ring $\Pb_m$ lexicographically  by using the following  
order of the variables: $x_{i, j} > x_{i', j'}$ if either $j > j'$ or $j = j'$ and $i > i'$. Define 
$e_{\pi, i} > e_{\rho, j}$ if $i < j$. For fixed $i$, identify a monomial $e_{\pi, i} \in \Fo{d_i}_m$ with a 
vector $(m, \pi (1),\ldots,\pi (d_i)) \in \N^{d_i+1}$ and order such monomials by using the 
lexicographic order on $\N^{d_i+1}$. For example, this implies that every $e_{\pi, i} \in \Fo{d_i}_m$ is 
smaller than any $e_{\tilde{\pi}, i} \in \Fo{d_i}_n$ if $m < n$.

Finally, for any  monomials $x^a e_{\pi, i}$ and $x^b e_{\tilde{\pi}, j}$ in $\Fo{d_i}_m$, define 
$x^ e_{\pi, i} > x^b e_{\tilde{\pi}, j}$ if either $e_{\pi, i} > e_{\tilde{\pi}, j}$ or 
$e_{\pi, i} = e_{\tilde{\pi}, j}$ and 
$x^a > x^b$ in $\Pb_m$. One checks that this gives indeed 
a monomial order on $\Fb$.
\end{ex}

We extend the earlier definition of $\OI$-divisibility in $\Fo{d}$ to  $\Fb$ by defining that a monomial $\mu \in \Fb$ is \emph{$\FIO$-divisible} by a monomial $\nu \in \Fb$ if $\mu$ and $\nu$ are in the same summand of $\Fb$ and $\nu$ $\OI$-divides $\mu$ in that summand.  
Equivalently, one has for two monomials 
$\mu, \nu$  of $\Fb$ that $\mu$ is $\OI$-divisible by $\nu$ if and only if $\mu \in \la \nu \ra_{\Fb}$.
It follows that every monomial order on $\Fb$ refines the partial order defined by $\FIO$-divisibility. More is true. 

Recall that a \emph{well-partial-order} on a set $S$ is a partial order $\le$ such that, for any infinite sequence $s_1,s_2, \ldots$ of elements in $S$, there is a pair of indices $i < j$ such that $s_i \le s_j$.

\begin{prop}
\label{prop:unique-min-element} \mbox{ } 
\begin{itemize}
\item[(a)] $\FIO$-divisibility is a well-partial-order on the set of monomials in $\Fb$. 

\item[(b)] Fix any monomial order $>$ on $\Fb$. Every non-empty set of monomials of $\Fb$ has a unique minimal element in the order $>$.
\end{itemize}
\end{prop}

\begin{proof}
This follows because the statements are true for the monomials in any summand $\Fo{d_i}$ of $\Fb$ by \cite[Proposition 6.2 and Corollary 6.5]{NR2}. 
\end{proof}

\begin{defn}
Let $>$ be a monomial order on $\Fb$.
Consider an element $q = \sum c_\mu \mu \in \Fb_m$ for some $m \in \N_0$ with monomials $\mu$ and coefficients
$c_{\mu} \in K$. If $q \neq 0$ we define its \emph{leading
monomial} $\lm (q)$ as the largest monomial $\mu$ with a non-zero coefficient $c_{\mu}$. This coefficient is called the \emph{leading coefficient}, denoted $\lc (q)$. The \emph{leading term} of $q$ is $\lt (q) = \lc (q) \cdot \lm (q)$.

For of a subset $E$ of $\Fo{d}$, we set $\lt (E) = \{ \lt (q) \; \mid \; q \in E\}$. 
\end{defn}

Recall that, for a subset $E$ of any $\FIO$-module $\M$, we denote by
$\la E \ra_{\M}$ the smallest $\FIO$-submodule of $\M$ that contains $E$. 

\begin{defn}
Fix a monomial order $>$ on $\Fb$,
and let $\M$ be an $\FIO$-submodule of $\Fb$.
\begin{enumerate}
\item The \emph{initial module} of $M$ is
\[
\ini (\M) = \ini_{>} (\M) = \la \lt (q) \s q \in \M \ra_{\Fb}.
\]
It is a submodule of $\Fb$.

\item A subset $B$ of $\M$  is a \emph{Gr\"obner basis} of $\M$ (with respect to $>$) if
\[
\ini (\M) = \la \lt (B) \ra_{\Fb}.
\]

\end{enumerate}
\end{defn}

We can now state the needed extension of \cite[Theorem 6.14]{NR2}. 

\begin{thm}
\label{thm:finite-G-basis}
Every $\FIO$-submodule of a finitely generated free $\FIO$-module  over $\Pb \cong (\XO{1})^{\otimes c}$ has a finite Gr\"obner basis (with respect to $>$).
\end{thm}

\begin{proof}
Given the above results, 
this follows as in \cite[Theorem 6.14]{NR2}. However, we can argue more directly as follows. By \cite[Theorem 6.15]{NR2}, every finitely generated $\OI$-module $\Fb$ over $\Pb$ is noetherian. Thus, the initial module of any submodule $\M$ of $\Fb$ is finitely generated. The elements of $\M$ whose leading terms give a finite generating set of $\ini (\M)$ form a Gr\"obner basis of $\M$. 
\end{proof}

Consider now a monomial order $<$ on $\Fb$ that induces an order on the monomials in $\Pb$ as described in \Cref{exa:mon-order}. For any integer $n \ge 0$, the order $<$ on $\Fb$ induces a monomial order $\preceq$ on the $\Pb_n$-module $\Fb_n$. We denote the initial module of any $\Pb_n$-submodule $N$ of $\Fb_n$ by $\ini_{\preceq} (N)$. These constructions are compatible. For simplicity, we write $\ini (\M)_n$ for the width $n$ component of $\ini (\M) = \ini_{<} (\M)$. 

\begin{lem}
     \label{lem:compare initial modules}
Let $<$ be a monomial order on $\Fb$ that induces an order on the monomials in $\Pb$ as used in \Cref{exa:mon-order}. If $\M$ is any submodule of $\Fb$, then on has, for every $n \ge 0$, 
\[
\ini (\M)_n = \ini_{\preceq} (\M_n). 
\]
\end{lem}

\begin{proof}
The definitions imply $\ini_{\preceq} (\M_n) \subset \ini (\M)_n$. In order to show the reverse inclusion we use the fact that $\ini (\M)_n$ is a monomial module. It follows that it suffices to prove that if $q \in \M$ and $\lt (q) = \lt (f) e_{\pi, k} \in \Fb_n$ for some $f \in \Pb$, then $f \in \Pb_n$.  This is a consequence of our assumption on the order of the monomials in $\Pb$. Indeed, if a variable $x_{i, j}$ with $j > n$ divides any monomial appearing in $f$, then this monomial is greater than any monomial in $\Pb_n$ in our ordering, which contradicts $\lt (f) \in \Pb_n$. Hence $f$ must be a polynomial in $\Pb_n$.  
\end{proof}

We now turn to graded modules and their Hilbert series. Thus, we assume from now on that $K$ is a field. Observe that $\Fb = \bigoplus_{i=1}^k \Fo{d_i}$ has $k$ generators. Assigning any integers as degrees of these generators turns $\Fb$ into a graded $\OI$-module over $\Pb$ with its standard grading. 

\begin{cor}
    \label{cor:compare Hilb series}
If  $\M$ is a graded submodule of $\Fb$, then 
\[
H_{\Fb/\M)} (s, t) = H_{\Fb/\ini (\M)} (s, t).
\] 
\end{cor}

\begin{proof}
Every module $\M_n$ is a graded submodule of the graded, finitely generated  free $\Pb_n$-module $\Fb_n$. Thus, it is well-known that $\M_n$ and $\ini_{\preceq} (\M_n)$ have the same Hilbert series. Hence \Cref{lem:compare initial modules} gives that $\Fb_n/\M_n$ and $\Fb_n/\ini (\M)_n$ also have the same Hilbert series. Now the claim follows. 
\end{proof}

Sometimes it is useful to adjust the grading of a graded $\OI$-module $\Nb$. 
For any integer $r$, we define $\Nb (r)$ to be the graded  $\OI$-module that is isomorphic to $\Nb$ as an $\OI$-module and  whose grading is defined by 
\[
[\Nb(r)_m]_j = [\Nb_m]_{r+j}. 
\]
Thus, passing to $\M(-r)$ for some sufficiently large integer $r$, it is harmless to assume that a finitely generated graded $\OI$-module has a generating set whose elements all have non-negative degrees. In this case, one says that $\M$ is generated in non-negative degrees. 

We are ready to establish the first main result of this section.

\begin{thm}
     \label{thm:shape Hilb f.g. OI-module}
If $\M$ is a finitely generated graded $\OI$-module over $(\XO{1})^{\otimes c}$ that is generated in non-negative degrees, then its equivariant Hilbert series is of the form
\[
H_{\M} (s, t) = \frac{g(s, t)}{(1-t)^a \cdot \prod_{j =1}^b [(1-t)^{c_j} - s \cdot f_j (t)]},
\]
where $a, b, c_j$ are non-negative integers with $c_j \le c$, \ $g (s, t) \in \Z[s, t]$, and each $f_j (t)$ is a polynomial in $\Z[t]$ satisfying $f_j (1) >  0$ and $f_j (0) = 1$.
\end{thm}

\begin{proof} 
Let $E = \{q_1,\ldots,q_k\}$ be a generating set of $\M$ with homogeneous elements $q_i \in \M_{d_i}$. It canonically determines a graded natural transformation 
\[
\Fb = \bigoplus_{i=1}^k \Fo{d_i}(-\deg q_i) \to \M 
\]
 whose image is $\M$. 
Thus, its kernel is a graded submodule $\Nb$ of $\Fb$ with  $\M \cong \Fb/\Nb$. 
Hence, choosing a suitable monomial order, \Cref{cor:compare Hilb series} gives that $\M$ and $\Fb/\ini (\Nb)$ have the same equivariant Hilbert series. Since $\ini (\Nb)$ is a monomial submodule there is an isomorphism of graded  $\OI$-modules 
\[
\Fb/\ini (\Nb)  \cong \bigoplus_{i=1}^k \Fo{d_i}(-\deg q_i)/\Nb^i 
\]
with monomial submodules $\Nb^i$ of $\Fo{d_i}(-\deg q_i)$. Observe that $\Fo{d_i}(-\deg q_i)/\Nb^i$ is equal to  $\big (\Fo{d_i}/\Nb^i (\deg q_i) \big)(- \deg q_i)$, and so 
\[
H_{\Fo{d_i}(-\deg q_i)/\Nb^i} (s, t) = t^{\deg q_i} H_{\Fo{d_i}/\Nb^i (\deg q_i)} (s, t)
\]
This gives 
\[
H_{\M} (s, t) = \sum_{i=1}^k  t^{\deg q_i} H_{\Fo{d_i}/\Nb^i (\deg q_i)} (s, t). 
\]
By assumption on $\M$, the degree of each $q_i$ is non-negative. Hence we conclude by applying  \Cref{thm:shape Hilb mon submodule} to each of the submodules $\Nb^i (\deg q_i)$. 
\end{proof}

\begin{rem}
     \label{rem:degree shifts Hilb}
(i) 
\Cref{thm:shape Hilb f.g. OI-module} is a full generalization and strengthening of one of the main results in \cite{NR}. Indeed, \cite[Proposition 7.2]{NR} may be viewed as result about the equivariant Hilbert series of graded quotients of $(\XO{1})^{\otimes c}$. 

(ii) 
The assumption on generator degrees is harmless. If $\M$ is finitely generated, then $\M(-k)$ is generated in non-negative degrees if $k \gg 0$. Since, 
\[
H_{\M (-k)} (s, t) = t^k \cdot H_{\M} (s, t)
\]
\Cref{thm:shape Hilb f.g. OI-module} gives for the equivariant Hilbert series of  any finitely graded module
\[
H_{\M} (s, t) = t^{-k} \cdot  \frac{g(s, t)}{(1-t)^a \cdot \prod_{j =1}^b [(1-t)^{c_j} - s \cdot f_j (t)]}
\]
for some integer $k \ge 0$. 
\end{rem}

\begin{ex} 
    \label{exa:ideal gen by one monomial}
If $\Ib$ is an ideal of $\Pb = (\XO{1})^{\otimes c}$ that is generated by one monomial then the equivariant Hilbert series of $\Pb /\Ib$ has been explicitly determined in \cite[Theorem 3.3]{GN}. In the special case, where $c =1$, let $\Ib$ be the ideal generated by $x_{i_1}^{a_1}  \cdots x_{i_r}^{a_r} \in \Pb_{i_r}$ with positive integers $a_1,\ldots,a_r$ and $i_1 < \cdots < i_r$. The equivariant Hilbert series of $\Pb/\Ib$ is
\[
H_{\Pb/\Ib} (s, t) =   \frac{g(s, t)}{(1-t)^{i_r -1} \cdot \prod_{j =1}^r [(1 - s \cdot (1 + t + \cdots + t^{a_j -1})]}, 
\]
where $g (s, t) \in \Z[s, t]$ is the polynomial with 
\[
g(s, t) \cdot (1-t-s) = (1-t)^{i_r -r} \prod_{j=1}^r (1-t-s + s t^{a_j}) - s^{i_r} t^{\sum_{j=1}^r a_j}.
\]
\end{ex}

There is a result analogous to \Cref{thm:shape Hilb f.g. OI-module}  for graded $\FI$-modules. It generalizes and strengthens  \cite[Theorem 7.8]{NR}. 

\begin{cor}
    \label{cor:shape Hilb f.g. FI-module}
If $\M$ is a finitely generated graded $\FI$-module over $(\XI{1})^{\otimes c}$ that is generated in non-negative degrees, then its equivariant Hilbert series is of the form
\[
H_{\M} (s, t) = \frac{g(s, t)}{(1-t)^a \cdot \prod_{j =1}^b [(1-t)^{c_j} - s \cdot f_j (t)]},
\]
where $a, b, c_j$ are non-negative integers with $c_j \le c$, \ $g (s, t) \in \Z[s, t]$, and each $f_j (t)$ is a polynomial in $\Z[t]$ satisfying $f_j (1) >  0$ and $f_j (0) = 1$.
\end{cor} 

\begin{proof}
We may consider $\M$ also an $\OI$-module over $(\XO{1})^{\otimes c}$. As such a module it is still finitely generated by \cite[Remark 3.17]{NR2}. Thus, \Cref{thm:shape Hilb f.g. OI-module} gives the result. 
\end{proof}

In the case where $\Pb$ is the smallest noetherian polynomial $\OI$- or $\FI$-algebra that is not equal to $K$ in every width, the above results can be considerably strengthened. 

\begin{cor}
    \label{cor: hilb for c=1}
Adopt the assumptions of \Cref{thm:shape Hilb f.g. OI-module}  or \Cref{cor:shape Hilb f.g. FI-module}. If $c = 1$, that is, $\Pb = \XO{1}$ or $\Pb = \XI{1}$, then the equivariant Hilbert series of 
$\M$ is of the form
\[
H_{\Fb/\M} (s, t) = \frac{g(s, t)}{(1-t)^a \cdot (1-t-s)^k \prod_{j =1}^b [1 - s \cdot (1 + t + \cdots + t^{e_j})]},
\]
where $a, b, k, e_j$ are non-negative integers and  \ $g (s, t) \in \Z[s, t]$.
\end{cor} 

\begin{proof}
This follows as \Cref{thm:shape Hilb f.g. OI-module} by invoking \Cref{thm:shape Hilb mon submodule c = 1} instead of \Cref{thm:shape Hilb mon submodule}. 
\end{proof}

The information about the denominators in the above Hilbert series allows us draw conclusions about the growth of classical invariants along a sequence of graded modules over noetherian polynomial rings determined by an $\OI$- or $\FI$-module. 

\begin{thm}
     \label{thm:growth invariants} 
Let $\M$ be a finitely generated graded $\OI$-module over $(\XO{1})^{\otimes c}$ or a finitely generated graded $\FI$-module over $(\XI{1})^{\otimes c}$. Then there are integers $A, B, M, L$ with $0 \le A \le c$, $M > 0$, and $L \ge 0$ such that, for all $n \gg 0$,
\[
\dim \M_n = A n + B
\]
and the  limit  of $\frac{\deg \M_n}{M^{n } \cdot  n^{L}}$ as $n \to \infty$ exists and is equal to a positive rational number. In particular, the limits
\[
\lim_{n \to \infty} \frac{\dim \M_n}{n} = A \quad \text{ and } \quad  \lim_{n \to \infty} \sqrt[n]{\deg \M_n} = M
 \]
 are non-negative integers. 
\end{thm}

\begin{proof}
This follows  from \Cref{thm:shape Hilb f.g. OI-module} and \Cref{cor:shape Hilb f.g. FI-module} exactly as \cite[Theorem 7.10]{NR} whose proof uses the mentioned special cases of the former results. 
\end{proof}

\begin{ex}
    \label{exa:asymptotic degree} 
(i) If $\Ib$ is an ideal of $(\XO{1})^{\otimes c}$ that is generated by one monomial $u \neq 0$, then 
$\lim_{n \to \infty} \frac{\dim ((\XO{1})^{\otimes c}/\Ib)_n}{n} = c-1$ by \cite[Corollary 3.8]{GN}. Moreover, if $c = 1$ and $u = x_{i_1}^{a_1}  \cdots x_{i_r}^{a_r}$ as in \Cref{exa:ideal gen by one monomial}, then $ \lim_{n \to \infty} \sqrt[n]{\deg (\XO{1}/\Ib)_n} = \max \{a_1,\ldots,a_r\}$ (see \cite[Corollary 3.8]{GN} for the general case $c \ge 1$). 

(ii) For an arbitrary monomial ideal $\Ib$ of $(\XO{1})^{\otimes c}$, the limit $\lim_{n \to \infty} \frac{\dim ((\XO{1})^{\otimes c}/\Ib)_n}{n}$ has been explicitly determined in \cite[Theorem 3.8]{LNNR2}.

(iii) Fix an integer $k \ge 1$ and let $\Ib$ be the $\OI$-ideal of $\Pb = \XO{1}$ that is generated in width $k+1$ by $\prod_{1 \le i <  j \le k+1} (x_i - x_j)$, that is, by the determinant  of the $(k+1) \times (k+1)$ Vandermonde matrix 
\[
\begin{bmatrix}
1 & 1  & \ldots & 1 \\
x_{1} & x_{2} & \ldots & x_{k+1}\\
x_{1}^{2} & x_{2}^{2} & \ldots & x_{k+1}^{2}
\vdots \\
x_{1}^{k} & x_{2}^{k} & \ldots & x_{k+1}^{k}
\end{bmatrix}. 
\]
For $n \ge k+1$, the ideal $\Ib_n$ is called a \emph{Vandermonde ideal} in \cite{WY}. These ideals are special cases of \emph{Specht ideals},  as introduced in \cite{Y} and further studied in \cite{SY}. Note that Vandermonde ideals are special cases of the ideals discussed in the introduction, where $\lambda = (k,k-1,\ldots,1)$ and $x_{i, j} = x_{1, j}$ for $1 \le i \le c =k$. In \cite[Theorem 1.1]{WY}, it is shown that $\Pb_n/\Ib_n$ is a reduced Cohen-Macaulay ring of dimension $k$ whose degree is given by the Stirling number $S(n, k)$ of the second kind. It follows that $ \lim_{n \to \infty} \sqrt[n]{\deg \Pb_n/\Ib_n} = k$.  
\end{ex}

The first part of \Cref{thm:growth invariants}  is also true for modules that are not necessarily graded. 

\begin{thm}
     \label{thm:growth of dim} 
Let $\M$ be a finitely generated  $\OI$-module over $(\XO{1})^{\otimes c}$ or a finitely generated $\FI$-module over $(\XI{1})^{\otimes c}$. Then there are integers $A, B$ such that, for all $n \gg 0$,
\[
\dim \M_n = A n + B. 
\]
\end{thm}

\begin{proof}
Present $\M$ as a quotient $\Fb/\Nb$ of a finitely generated free $\OI$-module $\Fb$. Using a suitable monomial order on $\Fb$, \Cref{lem:compare initial modules} gives $\ini (\Nb)_n = \ini_{\preceq} (\Nb_n)$. Hence, extending the arguments of \cite[Proposition 3.1(a)]{BC} from ideals 
to submodules of a free module, one gets 
$\dim \M_n=\dim (\Fb_n/\ini_{\preceq} (\Nb_n))$ for every $n$. Now 
the result follows from \Cref{thm:growth invariants}.
\end{proof} 


As preparation for the next result, the following observation will be useful. 

\begin{lem}
     \label{lem:powers of pol}
Let $f (t) \in \Z[t]$ be a polynomial with $f(0) \in \{0, 1\}$. For each  $n \in \N_0$, define integers $g_j (n)$ 
by 
\[
f(t)^n = \sum_{j = 0}^{n \cdot \deg f} g_j (n) t^j. 
\]
Thus, $g_j (n) = 0$ if $j < 0$ or $j > n \cdot deg f$. 
Then, for any fixed integer $j$, the number $g_j (n)$ is given by a polynomial in $n$ with rational coefficients whenever $n$ is sufficiently large. 
\end{lem}

\begin{proof} Let $f(t) = a_0 + a_1 t + \cdots + a_D t^D$. Using multi-index notation, for $k = (k_0,\ldots,k_D) \in \N_0^{D+1}$, we write $a^k = a_0^{k_0} \cdots a_D^{k_D}$ and $\binom{n}{k}$ for the multinomial coefficient 
\[
\binom{n}{k_0,\ldots,k_D} = \frac{n!}{k_0! \cdots k_D!} 
\]
with $|k| = k_0 + \cdots + k_D = n$. Thus, we get
\begin{equation}
      \label{eq:power of pol}
f(t)^n = \sum_{|k| = n} \binom{n}{k} a^k t^{k_1 + 2 k_2 + \cdots + D k_D}. 
\end{equation}
Consider any fixed integer $j \ge 0$. Observe that the number of tuples $(k_1,\ldots,k_D) \in \N_0^D$ with $k_1 + 2 k_2 + \cdots + D k_D = j$ is finite and independent of $n$. Hence, for any $n \gg 0$,  the coefficient of $t^j$ on the right-hand side of \Cref{eq:power of pol}, that is, $g_j (n)$, is a finite sum over these $D$-tuples  with summands of the form 
\[
 \frac{n!}{(n-(k_1 + \cdots + k_D))! \ k_1! \cdots k_D!} a_0^{n-(k_1 + \cdots + k_D)}  a_1^{k_1}\cdots a_D^{k_D}
\]
For any  fixed $(k_1,\ldots,k_D)$ contributing to the sum, the above multinomial coefficient is a polynomial in $n$ of degree $k_1 + \cdots + k_D$, which is bounded above by $j$, independent of $n$. Moreover, by assumption, we know $a_0 \in \{0, 1\}$. Hence, $a_0^{n-(k_1 + \cdots + k_D)}  a_1^{k_1}\cdots a_D^{k_D}$ does not depend on $n$, and our assertion follows. 
\end{proof}

Despite the exponential growth of the degrees of the modules $\M_n$, the vector space dimensions 
of their graded components of any fixed degree grow only polynomially. 

\begin{thm}
     \label{thm:polynomiality in fixed degree} 
Let $\M$ be a finitely generated graded $\OI$-module over $(\XO{1})^{\otimes c}$ or a finitely generated graded $\FI$-module over $(\XI{1})^{\otimes c}$. If $j$ is any fixed integer, then $\dim_K [\M_n]_j$ is given by a polynomial in $n$ with rational coefficients whenever $n$ is sufficiently large. 

More precisely, there are integers $a_1,\ldots,a_D$ (depending on $j$) such that, for any $n \gg 0$, 
\[
\dim_K [\M_n]_j = a_D \binom{n+D-1}{D-1} + \cdots + a_2 \binom{n+1}{1} + a_1,  
\]
where $a_D > 0$ unless $[\M_n]_j = 0$ for $n \gg 0$. 
\end{thm}

\begin{proof}
The second claim follows from the first one (see, e.g., \cite[Lemma 4.1.4]{BH}). 

To show the first claim, let $\M$ be finitely generated graded $\OI$-module. Possibly after applying a degree shift we may assume that $\M$ is generated in non-negative degrees. Thus, \Cref{thm:shape Hilb f.g. OI-module} applies to $\M$.  

Let $e \ge 0$ be an integer and let  $f(t) \in \Z[t]$ be a polynomial with $f(0) = 1$. 
We claim that the rational function ${\displaystyle  q(s, t) = \frac{(1-t)^e}{(1-t)^e - s \cdot f (t)}}$ can be written as 
\[
q (s, t) = \sum_{n \ge 0, j \ge 0} q_j (n) s^n t^j 
\]
with integers $q_j (n)$ that, for any fixed $j$, are given by a polynomial in $\Q[n]$ whenever $n \gg 0$. Indeed, using binomial series twice we compute
\begin{align*}
q (s, t) & = \frac{1}{1 - s \cdot \frac{ f (t)}{(1-t)^e}} = \sum_{n \ge 0}\frac{ f(t)^n}{(1-t)^{ en }} s^n \\[.5em]
& = \sum_{n \ge 0} f(t)^n \cdot \left [ \sum_{k \ge 0} \binom{k + e n -1}{k} t^k \right ] s^n.  
\end{align*}
Writing $f(t)^n = \sum_{j = 0}^{n \cdot \deg f} g_j (n) t^j$ as in \Cref{lem:powers of pol}, we get for $n \gg 0$, 
\[
q_j (n) = \sum_{k = 0}^j g_{k} (n) \binom{j-k + en -1}{j-k} 
\]
with integers $g_k (n)$ that, for fixed $k$, are eventually polynomial in $n$ by \Cref{lem:powers of pol}. Since $\binom{j-k + en -1}{j-k}$ is a polynomial in $n$ of degree $j-k \le j$, the claimed polynomial growth of $q_j (n)$ follows. 

Consider now a product of rational functions as studied in the above claim with integers $c_1,\ldots,c_b \ge 0$ and polynomials $f_1 (t),\ldots,f_b (t) \in \Z[t]$ with $f_j (0) = 1$. 
Applying the claim to each factor, we obtain 
\[
\prod_{j = 1}^b \frac{(1-t)^{c_j}}{(1-t)^{c_j} - s \cdot f _j(t)} =  \sum_{n \ge 0, k \ge 0} h_k (n) s^n t^k
\]
with integers $h_k (n)$ that, for any fixed $k$ grow eventually polynomially in $n$. By \Cref{thm:shape Hilb f.g. OI-module}, the equivariant Hilbert series of $\M$ differs from the above product only by a factor of $(1-t)^a g(s, t)$ for some integer $a$ and some $g(s, t) \in \Z[s, t]$, which implies the assertion about the growth of $\dim_K [\M_n]_j$ as a function in $n$. 

Finally, if $\M$ is a finitely generated graded $\FI$-module the claim follows as above by using \Cref{cor:shape Hilb f.g. FI-module} instead of \Cref{thm:shape Hilb f.g. OI-module}. 
\end{proof}


\section{Artinian Modules}
    \label{sec:artinian mod}

In this section we introduce Artinian graded $\OI$- or $\FI$-modules. For such modules, we strengthen some of the previous results.  We also derive consequences for  graded Betti numbers of the $\Pb_n$-modules $\M_n$. 

Recall that, for some fixed integer $c \ge 1$, we are considering $\OI$-modules over $\Pb = (\XO{1})^{\otimes c}$  or $\FI$-modules over $\Pb = (\XI{1})^{\otimes c}$. Throughout this section $K$ denotes any field. 

Classically, an Artinian module is defined by the stabilization of any descending sequence of submodules. The following example indicates that this is too restrictive for $\OI$-modules. 

\begin{ex}
     \label{exa:artinian of unbounded length} 
For $c = 1$ and some integer $a \ge 2$, let $\Ib \subset \Pb = \XO{1}$ be the ideal generated by $x_1^a$. Thus $\Ib_n = (x_1^a,\ldots,x_n^a) \subset K[x_1,\ldots,x_n] = \Pb_n$. 
Note that, for each $k \ge 1$, the monomial $x_1^{a-1} \cdots x_k^{a-1}$ is not in $\Ib_n$ for every $ \in \N$. For any $k \in \N$, let $\Jb^{(k)} \subset \Pb = \XO{1}$  be  the ideal that is generated in width $k$ by $x_1^{a-1} \cdots x_k^{a-1}$. This gives an infinite descending chain 
\[
\Pb/\Ib =  (\Jb^{(1)} + \Ib)/\Ib \supsetneq (\Jb^{(2)} + \Ib)/\Ib \supsetneq \cdots  
\]
that does not stabilize. Observe however that each of the modules $\Pb_n/\Ib_n$ is Artinian. 
\end{ex} 

In order to capture asymptotic properties of the modules $\M_n$ for large $n$, we propose. 

\begin{defn}
     \label{def:artinian} 
An $\OI$- or $\FI$-module $\M$ over $\Pb$ is said to be \emph{Artinian} if there is an integer $n_0$ such that  every $\Pb_n$-module $\M_n$ is Artinian whenever $n \ge n_0$. 
\end{defn}

Graded Artinian modules can be characterized by their  equivariant Hilbert series. 

\begin{thm}
     \label{thm:hilb artinian}
Let $\M$ be a  finitely generated graded $\OI$-module over $(\XO{1})^{\otimes c}$ or a finitely generated graded $\FI$-module over $(\XI{1})^{\otimes c}$. If $\M$ is  generated in non-negative degrees, then the following two conditions are equivalent: 
\begin{itemize}

\item[(i)] $\M$ is Artinian. 

\item[(ii)] The equivariant Hilbert series of $\M$ is of the form
\[
H_{\M} (s, t) = \tilde{g} (s, t) +   \frac{h(s, t)}{[f_1 (t) \cdots f_b (t)]^e \cdot \prod_{j =1}^b [1 - s \cdot f_j (t)]},    
\]
where $b, e$ are non-negative integers, $\tilde{g} (s, t)  \in \Q(t)[s]$,  $h (s, t) \in \Z[s, t]$ is not divisible by $1-t$ and  has degree less than $b$ considered as a polynomial in $s$ (and so $h (s, t) = 0$ if $b = 0$), and each $f_j (t)$ is a polynomial in $\Z[t]$ satisfying $f_j (1) >  0$ and $f_j (0) = 1$.
\end{itemize}

\end{thm}

In preparation of the proof, we establish the following observation. 

\begin{lem}
      \label{lem: t = 1}
For real numbers $a_1,\ldots,a_b$ and a non-zero polynomial $p(s) \in \R[s]$ with $\deg p(s) < b$, consider 
the formal power series expansion 
\[
\frac{p(s)}{\prod_{j =1}^b [1 - a_j s]}  = \sum_{n \ge 0} r_n s^n. 
\]
If each of $a_1,\ldots,a_b$ is at least 1, then $r_n \neq 0$ whenever $n \gg 0$. 
\end{lem}

\begin{proof}
Possibly after re-indexing the numbers $a_j$, we can write the given rational function $q(s)$ 
as 
\[
q (s) = \frac{p(s)}{\prod_{j =1}^m [1 - a_j s]^{b_j}}
\]
with positive integers $b_j$ and $a_1 > a_2 > \cdots > a_m \ge 1$. 
Using partial fractions, this can be rewritten as \[
q(s) = \sum_{j =1}^m \sum_{k = 1}^{b_j} \frac{c_{j, k}}{ [1 - a_j s]^k}
\] 
with real numbers $c_{j, k}$ satisfying $c_{j, b_j} \neq 0$ for $j = 1,\ldots,m$. It follows that 
\begin{align*}
q(s) & =  \sum_{j =1}^m \sum_{k = 1}^{b_j} c_{j, k} \sum_{n \ge 0}  \binom{n+k-1}{k-1} a_j^n s^n \\
& = \sum_{n \ge 0} \left [ \sum_{j = 1}^m  r_j (n) a_j^n \right ] s^n
\end{align*}
with polynomials $r_j (n) \in \R[n]$, 
\[
r_j (n) = \sum_{k = 1}^{b_j} \binom{n+k-1}{k-1} c_{j, k}. 
\]
Since $c_{j, b_j} \neq 0$, each $r_j (n)$ has degree $b_j -1$. In particular, $r_1 (n) \neq 0$. The last formula for $q(s)$ gives  $r_n = \sum_{j = 1}^m  r_j (n) \cdot  a_j^n$. If $a_1 = 1$, then $m$ must be one, and we get $r_n  \neq 0$ if $n \gg 0$. Otherwise, $a_1 > a_2 > \cdots > a_m \ge 1$ implies \[
\lim_{n \to \infty} \sqrt[n]{|r_n |} = \lim_{n \to \infty} \sqrt[n]{\left |\sum_{j = 1}^m  r_j (n) \cdot a_j^n \right | } = a_1,
\] 
which yields $r_n \neq 0$ whenever $n \gg 0$. 
\end{proof}

The conclusion of the lemma is not true without some assumption as shown by 
\[
\frac{1}{[1-s] \cdot [1 + s]} = \sum_{n \ge 0} s^{2 n}. 
\]

\begin{proof}[Proof of \Cref{thm:hilb artinian}] 
It is enough to establish the claim for an $\OI$-module, as it implies the result for $\FI$-modules. 
Let $\M$ be an $\OI$-module. 

First, assume $\M$ is Artinian. Consider the Hilbert series of $\M$ as described in 
\Cref{thm:shape Hilb f.g. OI-module}. We may assume that none of the irreducible factors of the 
denominator divides the numerator. Then the proof of \cite[Theorem 7.10]{NR} shows $c_1 = \cdots = c_b = 0$, that is,  
\[
H_{\M} (s, t) = \frac{g(s, t)}{(1-t)^a \cdot \prod_{j =1}^b [1 - s \cdot f_j (t)]}. 
\]
As a polynomial in $s$, the leading coefficient of $\prod_{j =1}^b [1 - s \cdot f_j (t)]$ is 
$r (t) = f_1 (t) \cdots f_b (t)$. Hence,  division with remainder  over $\Z[t]$ gives 
\begin{equation*}
r(t)^e \cdot g(s, t) = \tilde{h} (s, t) \cdot  \prod_{j =1}^b [1 - s \cdot f_j (t)] + (1-t)^l \cdot h(s, t)
\end{equation*}
with suitable integers $e, l \ge 0$ and polynomials $\tilde{h} (s, t), h (s, t) \in \Z[s, t]$ such that $(1-t)$ does not divide $h(s, t)$ and  the degree of $h (s, t)$ as a polynomial in $s$ is less than $b$. Note that $r (1) \neq 0$ because $f_j (1) \ge 1$ by \Cref{thm:shape Hilb f.g. OI-module}. 
Setting $\tilde{g} (s, t) = \frac{\tilde{h} (s, t)}{r(t)^e \cdot (1-t)^a}$, we get 
\begin{equation}
       \label{eq:after div with remainder}
H_{\M} (s, t) = \tilde{g} (s, t) +   \frac{h(s, t)}{r(t)^e \cdot (1-t)^{a - l} \cdot \prod_{j =1}^b [1 - s \cdot f_j (t)]}. 
\end{equation}
It remains to show $a - l = 0$. To this end consider the formal power series expansion of 
\[
 \frac{h(s, t)}{\prod_{j =1}^b [1 - s \cdot f_j (t)]} = \sum_{n \ge 0} h_n (t) s^n  
 \]
with polynomials $h_n (t) \in \Z[t]$. Since $f_j (1) \ge 1$ by \Cref{thm:shape Hilb f.g. OI-module} and $h(s, 1)$ is not the zero-polynomial by the choice of $l$, \Cref{lem: t = 1} gives $h_n (1) \neq 0$ 
whenever $n \gg 0$. Equation \Cref{eq:after div with remainder} yields for the Hilbert series of 
$\M_n$ with  $n \gg 0$, 
\[
H_{\M_n} (t) = \frac{h_n (t)}{r (t)^e \cdot (1-t)^{a - l}}. 
\]
By assumption, $\M_n$ has Krull dimension zero for such $n$. Hence $h_n (1) \neq 0$ and $r(1) \neq 0$ imply $a - l = 0$, as desired. 

Second, assume conversely that $\M$ has an equivariant Hilbert series as stated in (ii). Using  binomial series, one gets 
\[
\frac{h(s, t)}{\prod_{j =1}^b [1 - s \cdot f_j (t)]} = h(s, t) \cdot  \prod_{j = 1}^b \left[ \sum_{n \ge 0}  f_j (t) ^n s^n \right ] = \sum_{n \geq 0} h_n (t) s^n 
\]
with suitable polynomials $h_n (t) \in \Z[t]$. 
Setting again $r (t) = f_1 (t) \cdots f_b (t)$, we have $r (1) \neq 0$. Moreover, the assumption gives for the Hilbert series of $\M_n$ with $n \gg 0$, 
\[
H_{\M_n} (t) = \frac{h_n (t)}{r (t)^e}. 
\]
Since $r(1) \neq 0$, it follows that $\M_n$ has Krull dimension zero, i.e., it is Artinian. 
\end{proof}

The above results have consequences for graded Betti numbers. Let $P$ be a polynomial ring over $K$ with finitely many variables and its standard grading. Then every  finitely generated graded $P$-module $M$ has a graded minimal free resolution of the form 
\[
0 \to \bigoplus_{j} P^{\beta_{k, j}} (-j) \to \cdots \to \bigoplus_{j} P^{\beta_{1, j}} (-j) \to \bigoplus_{j} P^{\beta_{0, j}} (-j) \to M \to 0, 
\]
where each of the appearing free modules is finitely generated. Moreover, the numbers $\beta_{i, j} =  \beta_{i, j}^P (M)$ are uniquely determined by $M$ because 
$\beta_{i, j}^P (M) = \dim_K [\Tor_i^P (M, K)]_j$ and are called the \emph{graded Betti numbers} of $M$.  
The \emph{Castelnuovo-Mumford regularity} of $M$ is 
\[
\reg M = \max \{ j - i \; \mid \;  [\Tor_i^P (M, K)]_j \neq 0 \}. 
\]

The graded Betti numbers of $M$  are often displayed in the \emph{Betti table} of $M$ whose $(i, j)$-entry is $\beta_{i, i+j}^{P} (M)$. For example, consider the ideal $I = \la x^3, x^2 y, x y z, y^4 \ra$  of $P = K[x, y, z]$. The minimal graded free resolution of $M = P/I$ has the form
\[
0 \to P (-7) \to P^2 (-6) \oplus P^2 (-4) \to P (-4) \oplus P^3 (-3) \to P \to M \to 0. 
\]
In particular, one has $\reg M = 4$ and $\pd M = 3$. The Betti table of $M$ is \\[-8pt]
\begin{center}
\begin{minipage}{6.5cm}
\begin{verbatim}
        0    1    2    3
-------------------------
 0:     1    -    -    -
 1:     -    -    -    -
 2:     -    3    2    -
 3:     -    1    -    -
 4:     -    -    2    1 
\end{verbatim} 
\end{minipage} \\[10pt]
\end{center}

The following result shows in particular that the entries in the Betti tables of  $\M_n$ in a fixed position vary eventually polynomially with $n$.

\begin{thm}
    \label{thm:Betti numbers} 
Let $\M$ be a finitely generated graded  $\OI$-module over $\Pb = (\XO{1})^{\otimes c}$ or a finitely generated graded  $\FI$-module over $\Pb = (\XI{1})^{\otimes c}$. If $i \ge 0$ is any fixed integer, then  there are integers $j_1 (\M, i) < \cdots < j_l (\M, i)$ and polynomials $p_1 (t),\ldots,p_l(t) \in \Q[t]$ such that, for any $n \gg 0$, 
one has for the graded Betti numbers of $\M_n$
\[
\beta_{i, j}^{\Pb_n} (\M_n) = \begin{cases}
0 & \text{ if } j \notin \{j_1 (\M, i),\ldots,j_l (\M, i)\} \\
p_k (n) & \text{ if $ j = j_k (\M, i)$ for some $k \in [l]$}. 
\end{cases}
\]
\end{thm}

\begin{proof}
Let $\M$ be an $\OI$-module. Fix $i \in \N_0$. By \cite[Lemma 7.4]{NR2},  there is a finitely generated graded $\OI$-module $\Nb$ over $\Pb$ with $\Nb_n = \Tor_i^{\Pb_n} (\M_n, K)$ for each $n$. The vanishing statement for the graded Betti numbers follows by \cite[Theorem 7.7]{NR2}. 
The other claim is an immediate consequence of \Cref{thm:polynomiality in fixed degree} because $\beta_{i, j}^{\Pb_n} (\M_n) = \dim_K [\Nb_n]_j$. 

The argument is analogous if $\M$ is an $\FI$-module. 
\end{proof}

\begin{rem}
Following \cite{Ramos}, this result has consequences in the study of configuration spaces of graphs. If $G$ is any finite graph, then its $j$-th configuration space is the topological space 
$U\Fc_j (G) = \{(y_1,\ldots,y_j) \in G^j \; \mid \; y_i \neq y_k  \}/\Sym (j)$.  For $q \in \N_0$, the total $q$-th homology group of $G$ is $\Hc_q (G) = \bigoplus_{j \ge 0} H_q (U\Fc_j (G))$. It can be 
given the structure of a finitely generated graded module over a polynomial ring $A_G$ whose variables are indexed by the vertices of $G$ (see \cite{ADK}). Fix now two finite graphs $G$ and $H$ along with vertices $v_G \in G$ and $v_H \in H$. Define $G_n = G \bigvee  H^{\vee n}$  by wedging $G$ with $H$ $n$-times. If the number of edges  of $G_n$ is linear in $n$ for $n \gg 0$, then Ramos showed that, for any fixed $q \in \N_0$,  the assignment $[n] \mapsto \Hc_q (G_n)$ defines a finitely generated graded $\FI$-module  over $(\XI{1})^{\otimes c}$ for some $c$ (see \cite[Theorem 4.4 and Proposition 3.17]{Ramos}. Hence, \cite[Corollary 4.6]{Ramos} may be viewed as an instance of \Cref{thm:Betti numbers}. 
\end{rem}

For particular Artininian modules, the Castelnuovo-Mumford regularity becomes asymptotically constant. 

\begin{cor}
    \label{cor:invariants if finite support} 
Let $\M$ be a finitely generated graded  $\OI$-module over $\Pb = (\XO{1})^{\otimes c}$ or a finitely generated graded  $\FI$-module over $\Pb = (\XI{1})^{\otimes c}$. If there is an integer $\delta$ such that, for $n \gg 0$, one has $[\M_n]_j = 0$ whenever $j > \delta$, then there is an integer $C$ such that $\reg \M_n = C$ if $n \gg 0$, unless $\M_n = 0$ whenever $n \gg 0$. 
\end{cor}

\begin{proof}
By assumption, $\M_n$ is Artinian if $n \gg 0$. Hence it is well-known (see, e.g., \cite[Corollary 4.4]{Ei}) that 
 $\reg (\M_n) = \max \{j \in \Z \; \mid \; [\M_n]_j \neq 0 \} \leq \delta$. Now \Cref{thm:polynomiality in fixed degree} gives the claim. 
\end{proof}


\section{Concluding Remarks and Conjectures} 
     \label{sec:concluding remarks}

We begin by pointing out that all of the above results can be applied to a finitely graded $\OI$-module over $(\XO{1}_K)^{\otimes c}$ if $K$ is not necessarily a field, but  a standard graded noetherian algebra over a field $k$. 
The next goal is to describe and to explain a finite algorithm that computes the equivariant Hilbert series of a finitely generated graded $\OI$-module over $(\XO{1})^{\otimes c}$. This extends work in \cite{KLS} for ideals of $(\XO{1})^{\otimes c}$. Then we discuss questions that arise from the results in this paper. We hope they inspire future investigations.

Above we systematically avoided referring to the base ring $K$ to simplify notation. However, we will consider two base rings in the following statement. 

\begin{thm}
    \label{thm:base ring} 
Suppose $K$ is a standard graded finitely generated algebra over a field $k$. 
Let $\M$ be a finitely generated graded  $\OI$-module over $\Pb = (\XO{1}_K)^{\otimes c}$ Let $\Nb$ be the $\OI$- module over $\Pb$ obtained from $\M$ by setting $\Nb_0 = 0$ and $\Nb_n = \M_n$ if $n \ge 1$. Then there is an integer $r \ge 0$ such that $\Nb$ is a finitely generated graded $\OI$-module over $(\XO{1}_k)^{\otimes (c+r)}$. In particular, 
$\M$ has an equivariant Hilbert series of the form
\[
 \sum_{n \ge 0, j \in \Z} \dim_k [\M_n]_j s^n t^j  = t^l \frac{g(s, t)}{(1-t)^a \cdot \prod_{j =1}^b [(1-t)^{c_j} - s \cdot f_j (t)]},
\]
where $a, b, c_j$ are non-negative integers with $c_j \le c+r$, \ $l \in \Z$, \ $g (s, t) \in \Z[s, t]$, and each $f_j (t)$ is a polynomial in $\Z[t]$ satisfying $f_j (1) >  0$ and $f_j (0) = 1$, 
and the Krull dimension $\dim \M_n$ is given by linear function in $n$ for $n \gg 0$. 
\end{thm}

\begin{proof}
By assumption $K$ is isomorphic to $k[y_1,\ldots,y_r]/J$, where $J$ is a proper homogeneous ideal of the polynomial ring $k[y_1,\ldots,y_r]$ and every variable $y_i$ has degree one. Hence, using \Cref{exa:reduction to field}, it follows that $\Nb$ is a finitely generated graded  $\OI$-module over $(\XO{1}_K)^{\otimes (c+r)}$. Therefore, \Cref{thm:shape Hilb f.g. OI-module} and \Cref{thm:growth invariants} apply to $\Nb$. Note that the factor $t^l$ accounts for a degree shift that may be needed to obtain from $\Nb$ a module whose generators have non-negative degrees. 
\end{proof}

An analogous statement is true for any finitely generated graded  $\FI$-module over $\Pb = (\XI{1}_K)^{\otimes c}$. We leave the details to the reader. 
\smallskip 

For the remainder of this section, we return to the set-up where $K$ is a field and consider polynomial $\OI$- or $\FI$-algebras over $K$. 
In order to describe an algorithm for computing equivariant Hilbert series we first recall some background material. A \emph{(deterministic) finite 
automaton} on an alphabet $\Sigma$  is a $5$-tuple $\cA = (P, \Sigma, \delta, p_0, F)$ consisting  
of a  finite set $P$ of \emph{states},  an \emph{initial state} $p_0 \in P$, a set $F \subset P$ of 
\emph{accepting states}  and a \emph{transition map} $\delta \colon D \to P$, where $D$  is  some 
subset of $P \times \Sigma$.  The automaton $\cA$ \emph{recognizes} or \emph{accepts} a word 
$w=a_1a_2\dots a_s \in \Sigma^*$ if there is a sequence of states $r_0,r_1,\dots, r_s$ satisfying 
$r_0 = p_0$, $r_s \in F$ and 
\[
r_{j+1} = \delta (r_j, a_{j+1}) \quad \text{whenever } 0 \le j < s. 
\] 
The automaton $\cA$ \emph{recognizes a formal language} $\cL \subset \Sigma^*$ if $\cL$ is precisely the set of words in $\Sigma^*$ that are accepted by $\cA$. By \cite[Theorems 3.4 and 3.7]{HU}, a formal language is regular if and only if it is recognizable by a finite automaton. Their generating series are computable. 

Indeed, suppose $\cL$ is a formal language on $\Sigma$ that is recognized by a finite automaton $\cA = (P, \Sigma, \delta, p_0, F)$, where  $P$ has $N$ elements $p_0,\ldots,p_{N-1}$. For every letter $a \in \Sigma$ define a  $0-1$ matrix $M_{\cA, a}$ of size $N \times N$. Its entry at position $(i , j)$ is 1 precisely if there is a transition $\delta (p_j, a) = p_i$. Let $\mathbf{e}_i \in \mathbb{K}^N$ be the standard basis vector corresponding to   state $p_{i-1}$.  Let $\mathbf{u}=\sum\limits_{p_{i-1} \in F} \mathbf{e}_i \in \mathbb{K}^N$ be the sum of the basis vectors corresponding to the accepting states. Then, for any word $w=a_1\dots a_d$ with $a_i \in \Sigma$, one has 
\[
\mathbf{u}^{T} M_{\cA,a_d}\dots A_{\cA,a_1} \mathbf{e}_1  =\begin{cases} 
      1 & \text{ if } \cA \text{ accepts } w \\
      0 & \text{ if }  \cA \text{ rejects } w.
   \end{cases}
   \]
Let $\rho: \Sigma^*\rightarrow \Mon (K[s_1,\ldots,s_k])$ be any \emph{weight function}, that is, a monoid homomorphism. 
It follows (see, e.g, \cite[Section 4.7]{St}): 
\begin{align} 
     \label{eq:formule for generating function} 
P_{\cL,\rho} (s_1,\dots,s_k) & =\sum_{w\in\cL} \rho(w)  
=  \sum_{d\ge 0} \ \sum_{a_1,\ldots,a_d \in \Sigma}  \mathbf{u}^T \left( \rho (a_1\dots a_d)  M_{\cA,a_d}\dots A_{\cA,a_1} \right ) \mathbf{e}_1 \nonumber \\
& =\sum_{d \ge 0} \mathbf{u}^T\left(\sum_{a\in\Sigma} \rho(a) M_{\cA,a} \right)^d \mathbf{e}_1 
= \mathbf{u}^T \left( \id_{N}-\sum_{a\in\Sigma} \rho(a) M_{\cA,a} \right)^{-1}\mathbf{e}_1. 
\end{align}
Thus, the generating function $P_{\cL,\rho} (s_1,\dots,s_k)$ can be explicitly computed from the  automaton $\cA$ using linear algebra.  

Now we are ready to state the algorithm. Afterwards we provide additional explanations for some of its steps. 

\begin{alg}
      \label{alg:equiv Hilbert series} 
      \mbox{}
      
Input: A graded $\OI$-module $\M$ over $\Pb = (\XO{1})^{\otimes c}$ given by a finite generating set $E$ of homogeneous elements and a finite set of relations $S$. 

Output: Rational function for the equivariant Hilbert series of $\M$. 
\begin{enumerate}
\item Let $E = \{q_1,\ldots,q_k\}$  with $q_i \in \M_{d_i}$. By assumption, there is a graded surjection 
\[
\Fb = \bigoplus_{i=1}^k \Fo{d_i}(-\deg q_i) \to \M 
\]
whose kernel $\Nb$ is generated by $S$. Thus, $\M \cong \Fb/\Nb$. 

\item Fix a monomial order  $<$ on $\Fb$ that induces an order on the monomials in $\Pb$ as in \Cref{exa:mon-order}.  Compute a Gr\"obner basis of $\Nb$ and so its initial module $\ini_{<} (\Nb)$. It determines monomial submodules $\Nb^i (\deg q_i) \subset \Fo{d_i}$ such that 
\[
\Fb/\ini_{<} (\Nb)  \cong \bigoplus_{i=1}^k \Fo{d_i}(-\deg q_i)/\Nb^i.  
\]

\item For every $i \in [k]$, construct first a regular expression for the language $\mu_{d_i}^{-1} (\Mon (\Nb^i (\deg q_i)))$ and then a finite automaton $\cA_i$ that accepts this language. 

\item Compute the equivariant Hilbert series of each $\Nb^i (\deg q_i) \subset \Fo{d_i}$ by Formula \eqref{eq:formule for generating function}, using the automaton $\cA_i$ and the weight function from \Cref{thm:rational hilb monomial submod} (see \eqref{eq:weight function}). 

\item Return  
\begin{align}
     \label{eq:hilb from automaton} 
H_{\M} (s, t) =\sum_{i=1}^k  t^{\deg q_i} \cdot \left [  \frac{s^{d_i} (1-t)^c}{[(1-t)^c - s]^{d_i+1}}  - H_{ \Nb^i (\deg q_i)} (s, t) \right ]. 
\end{align}
\end{enumerate}
\end{alg}

\begin{rem}
      \label{rem:algorithm} 
Step (ii): Computation of a Gr\"obner basis can be done by adapting the equivariant version of   Buchberger's algorithm to the setting of $\OI$-modules (see the Division Algorithm in \cite[Definition 6.11]{NR2} and \cite{Draisma, HKL}). 

Step (iii): Every $\Nb^i (\deg q_i) \subset \Fo{d_i}$ has finitely many monomial generators. For any such generator $q$, it is not difficult to transform the formula for $\mu_{d_i}^{-1} (\Mon (\la q \ra) )$ given by  Propositions \ref{prop:bijection} and  \ref{prop:inverse of monomial} to a regular expression that is constructed algorithmically from $q$. Using this expression for every generator, there is an algorithm to compute a regular expression for the union of these languages, that is, for $\cL_i = \mu_{d_i}^{-1} (\Mon (\Nb^i (\deg q_i)))$ (see Identity \eqref{eq:union monomials}). Then one passes algorithmically to a finite automaton $\cA_i$ that recognizes $\cL_i$ (see \cite[Chapter 2]{HU}). 

Step (v): By \Cref{lem:compare initial modules},  $\M$ and $\Fb/\ini_{<} (\Nb)$ have the same equivariant Hilbert series. 
Hence, 
\begin{align*}
H_{\M} (s, t) & = \sum_{i=1}^k  t^{\deg q_i} H_{\Fo{d_i}/\Nb^i (\deg q_i)} (s, t) \\
& = \sum_{i=1}^k  t^{\deg q_i} \cdot \left [  \frac{s^{d_i} (1-t)^c}{[(1-t)^c - s]^{d_i+1}}  - H_{ \Nb^i (\deg q_i)} (s, t) \right ],  
\end{align*}
as stated in \Cref{eq:hilb from automaton}. Note that we used \Cref{prop:hilb free OI-mod} for the second equality. 
\end{rem} 


Finally, we discuss again free resolutions of modules over polynomial rings.  
In \Cref{sec:artinian mod} we focussed mostly on graded Betti numbers, that is, on the entries of Betti tables. Our results and previous results in special cases suggest also an expectation on the sizes of Betti tables. Consider first the number of rows.  

\begin{conj}
    \label{conj:reg growth}
If $\M$ is a finitely generated  graded $\OI$-module over $\Pb = (\XO{1})^{\otimes c}$ or a finitely generated graded graded $\FI$-module over $\Pb = (\XI{1})^{\otimes c}$, then the Castelnuovo-Mumford regularity of $\M_n$ is eventually a linear function in $n$, that is, there are  integers $a, b$ such that 
\[
\reg \M_n = a n + b \quad \text{ whenever } n \gg 0. 
\]
\end{conj}

In the case of an $\OI$-ideal this specializes to \cite[Conjecture 1.1]{LNNR}. 

\begin{rem}
(i) Using \cite[Theorem 1.2]{HT}, one can obtain a coarse upper linear bound $\reg \M_n \le a n + b$ if $n \gg 0$ for some integers $a, b$. Extending methods in \cite{LNNR}, it is possible to establish a better  linear upper bound (see \cite{NR3}). 

(ii) Besides \Cref{cor:invariants if finite support}, the strongest evidence for \Cref{conj:reg growth} is known in the case of ideals (see \cite{LNNR}). If $c = 1$ and $\M= \Ib$ is a monomial $\FI$-ideal of $\XI{1}$, then \Cref{conj:reg growth} has been shown independently by Murai \cite{Mu} and by Raicu \cite{Ra19}. The conjecture is open for monomial $\OI$-ideals of $\XO{1}$. 
Note that 
Example 5.6 by Hop Nguyen in \cite{Mu} shows that the Betti tables of the ideals $\Ib_n$ are more complex for a monomial $\OI$-ideal than for a monomial $\FI$-ideal.  
\end{rem}

Next, we consider the number of columns in the Betti table of $\M_n$ as $n$ varies. The expectation is also meaningful for modules that are  not necessarily graded. 

\begin{conj}
    \label{conj: pd growth}
If $\M$ is a finitely generated   $\OI$-module over $\Pb = (\XO{1})^{\otimes c}$ or a finitely generated  $\FI$-module over $\Pb = (\XI{1})^{\otimes c}$, then the projective dimension of $\M_n$ is eventually a linear function in $n$, that is, there are  integers $A, B$ such that 
\[
\pd_{\Pb_n} \M_n = A n + B \quad \text{ whenever } n \gg 0. 
\]
\end{conj}

In the case of an $\OI$-ideal this specializes to \cite[Conjecture 1.3]{LNNR2}. 

\begin{rem}
(i) Extending work in \cite{LNNR2}, $\pd_{\Pb_n} \M_n$ can be bounded above and below by linear functions in $n$ (see \cite{NR3}). 

(ii) If $c = 1$ and $\M= \Ib$ is a monomial $\FI$-ideal of $\XI{1}$, then \Cref{conj:reg growth} is true  due to Murai \cite{Mu} and to Raicu \cite{Ra19}.
\end{rem}

Notice that in general, for  an $\OI$- or $\FI$-module $\M$,  invariants of $\M_n$ such as Castelnuovo-Mumford regularity, projective dimension or graded Betti numbers, grow unboundedly with $n$. Thus, recent boundedness results in \cite{AHo, Draisma-noeth, ESS-big, ESS}  do not apply directly to the categories of finitely generated $\OI$- or $\FI$-modules over a corresponding noetherian polynomial algebra. 


\end{document}